\documentclass[11pt]{amsart}

\usepackage{amsmath, amsthm, amssymb, amsfonts, enumerate}
\usepackage{enumitem}
\usepackage[colorlinks=true,linkcolor=blue,urlcolor=blue]{hyperref}
\usepackage{dsfont}
\usepackage{color}
\usepackage{geometry}
\usepackage{todonotes}
\usepackage{enumerate}

\newtheorem{theorem}{Theorem}[section]
\newtheorem{remark}[theorem]{Remark}
\newtheorem{assumption}[theorem]{Assumption}
\newtheorem{lemma}[theorem]{Lemma}

\newtheorem{definition}{Definition}[]

\theoremstyle{plain}

\def \R{\mathbb{R}}

\def \eps{\varepsilon}
\usepackage[utf8]{inputenc}
\usepackage[english]{babel}
\usepackage{indentfirst}
\usepackage{graphicx}
\usepackage{pdfpages}
\usepackage[title]{appendix}
\usepackage{fancyhdr}

\usepackage{enumitem}
\usepackage{amsmath}
\usepackage{amssymb}
\usepackage{bm}
\usepackage{mathrsfs}
\usepackage{amsthm}
\usepackage{comment}
\usepackage{hyperref}
\usepackage{dsfont} 
 
\usepackage{mathtools}
\DeclareMathOperator*{\argmin}{arg\,min}

\DeclareMathOperator*{\esssup}{ess\,sup}

\title[ Nonzero-Sum Submodular Monotone-Follower Games]{Nonzero-Sum Submodular Monotone-Follower Games:\\ Existence and Approximation of Nash Equilibria}
\author[Dianetti]{Jodi Dianetti}
\author[Ferrari]{Giorgio Ferrari}

\address{J.~Dianetti: Center for Mathematical Economics (IMW), Bielefeld University, Universit\"atsstrasse 25, 33615, Bielefeld, Germany}
\email{\href{mailto:giorgio.ferrari@uni-bielefeld.de}{jodi.dianetti@uni-bielefeld.de}}
\address{G.~Ferrari: Center for Mathematical Economics (IMW), Bielefeld University, Universit\"atsstrasse 25, 33615, Bielefeld, Germany}
\email{\href{mailto:giorgio.ferrari@uni-bielefeld.de}{giorgio.ferrari@uni-bielefeld.de}}
\date{\today}

\numberwithin{equation}{section}

\begin{document}
%\tableofcontents
\begin{abstract}
We consider a class of ${N}$-player stochastic games of multi-dimensional singular control,  in which each player faces a minimization problem of monotone-follower type with submodular costs. We call these games \emph{monotone-follower games}. In a not necessarily Markovian setting, we establish the existence of Nash equilibria. 
Moreover, we introduce a sequence of approximating games by restricting, for each $n\in \mathbb{N}$, the players' admissible strategies to the set of Lipschitz  processes with Lipschitz constant bounded by $n$. We prove that, for each $n\in \mathbb{N}$, there exists a Nash equilibrium of the approximating game and that the sequence of Nash equilibria converges, in the Meyer-Zheng sense, to a weak (distributional) Nash equilibrium of the original game of singular control. As a byproduct, such a convergence also provides approximation results of the equilibrium values across the two classes of games. We finally show how our results can be employed to prove existence of open-loop Nash equilibria in an $N$-player stochastic differential game with singular controls, and we propose an algorithm to determine a Nash equilibrium for the monotone-follower game. 
\end{abstract}
\maketitle


%%%%%%%%%%%%%%%%%%%%%%%%%%%%%%%%%%%%%%%%%%%%%%%%%%%%%%%%%%%%%%%
%%%              INTRODUCTION
%%%%%%%%%%%%%%%%%%%%%%%%%%%%%%%%%%%%%%%%%%%%%%%%%%%%%%%%%%%%%%%

\begin{section}
{Introduction}

We consider a class of stochastic $N$-player games over a finite time-horizon in which each  player, indexed by $i=1,...,N$, faces a multi-dimensional singular stochastic control problem of monotone-follower type. On a complete probability space, consider a multi-dimensional c\`adl\`ag (i.e., right-continuous with left limits) process $L$ and, for $i=1,...,N$, multi-dimensional continuous semimartingales $f^i$ with nonnegative components.  Denote by $\bar{\mathbb{F}}_+^{f,L}$ the right-continuous extension of the filtration generated by $f=(f^1,...,f^N)$ and $L$, augmented by the sets of zero probability. We call \emph{monotone-follower game}  the game in which each player $i$ is allowed to choose a multi-dimensional control $A^i $ in the set of \emph{admissible strategies}
\begin{equation*}
\mathcal{A}: = \left\{ \text{$\bar{\mathbb{F}}_+^{f,L}$-adapted processes with nondecreasing, nonnegative and c\`adl\`ag components}  \right\},
\end{equation*}
in order to minimize the cost functional 
\begin{align*}
\mathcal{J}^i(A^i,A^{-i} ) := \mathbb{E} \bigg[ \int_0^T h^i(L_t, A_t^i,A_t^{-i})\, dt +g^i(L_T,A_T^i,A_T^{-i}) + \int_{[0,T]} f_t^i \, dA_t^i   \bigg],
\end{align*}
where $A^{-i}:=(A^j)_{j\ne i}$. Here $T<\infty$ and $h^i$ and $g^i$ are suitable nonnegative convex cost functions.

Next, we introduce a sequence of approximating games with regular controls in the following way.
For each $n\in \mathbb{N}$, define the \emph{$n$-Lipschitz game} as the game in which players are restricted to pick a Lipschitz control in the set of \emph{admissible $n$-Lipschitz strategies}
\begin{equation*}
\mathcal{L}(n) = \left\{ A \in \mathcal{A} \, | \, A  \text{ is Lipschitz with Lipschitz constant smaller that }  n \text{ and } A_0=0 \right\},
\end{equation*}
in order to minimize the cost functionals $\mathcal{J}^i$.

Our main contributions are the following.
\begin{enumerate} 
    \item  Under submodularity conditions on the functions $h^i$ and $g^i$, we establish the existence of Nash equilibria for the monotone-follower and the $n$-Lipschitz games.
    \item We show  connections across these two classes of games. In particular: 
    \begin{enumerate}[label=(\roman*)]
         \item Any sequence obtained by choosing, for each $n \in\mathbb{N}$, a Nash equilibrium of the $n$-Lipschitz game is relatively compact in the Meyer-Zheng topology, and any accumulation point of this sequence is the law of a \emph{weak Nash equilibrium} of the monotone-follower game (see Definition \ref{def_weak_nash} below). That is, any accumulation point is a Nash equilibrium on a suitable probability space on which are defined processes $\bar{f}$ and $\bar{L}$ such that their joint law coincides with the joint law of $f$ and $L$.
        \item The $N$-dimensional vector whose components are the expected costs associated to any weak Nash equilibrium obtained through the previous approximation is a \emph{Nash equilibrium payoff}. Moreover, for each $\eps>0$, there exist $n_\eps \in \mathbb{N}$ large enough and a Nash equilibium of the $n_\eps$-Lipschitz game which is an $\eps$-Nash equilibrium of the monotone-follower game. 
     \end{enumerate}
\end{enumerate}
Furthermore, we provide applications of our results to deduce existence of Nash equilibria for a class of stochastic differential games with singular controls and non-Markovian random costs. Also, in the spirit of \cite{To}, we construct an algorithm to determine a Nash equilibrium of the monotone-follower game. 

To the best of our knowledge, general existence and approximation results for Nash equilibria in $N$-player non-Markovian stochastic games of multi-dimensional singular control appear in this paper for the first time.

\subsection{Background literature.} 
A singular stochastic control problem appears for the first time in \cite{Bather&Chernoff67}, where the problem of controlling the motion of a spaceship has been addressed. Later on, examples of solvable singular stochastic control problems have been studied in \cite{Benes&Shepp&Witsenhausen80}. 

Singular stochastic control problems of monotone-follower type have been introduced and studied in \cite{Karatzas81} and \cite{Karatzas&Shreve84}. A monotone-follower problem is the problem of tracking a stochastic process by a nondecreasing process in order to optimize a certain performance criterion. Since then, this class of problems has found many applications in economics and finance (see \cite{Bank&Riedel01},    \cite{Chiarolla&Ferrari14}, \cite{Davis&Norman90}, among many others), operations research (see, e.g., \cite{Guo&Kaminsky&Tomecek&Yuen11} and \cite{Harrison&Taksar83}), queuing theory (see, e.g., \cite{Krichagina&Taksar92}), mathematical biology (see, e.g., \cite{Alvarez&Sheep98}), aerospace engineering (see, e.g., \cite{Menaldi&Taksar89}), and insurance mathematics (see \cite{Lokka&Zervos08}, among others). 

The literature on singular stochastic control problems experienced results on existence of minima (or maxima) (see  \cite{Dufour&Miller04} and \cite{Haussmann&Suo95}, among others), characterization of the optimizers through first order conditions (see, e.g., \cite{Bank05},  \cite{Bank&Riedel01} and \cite{Cadenillas&Haussman94}), as well as connections to optimal stopping problems (see, e.g., \cite{Karatzas&Shreve84} or the more recent \cite{Boetius&Kohlmann98}) and to constrained backward stochastic differential equations \cite{Bouchardetal}. We also mention the recent work \cite{LZ}, as their version of the monotone-follower problem is  the single-agent version (in weak formulation) of our game.  

The number of contributions on games of singular controls is still quite limited (see \cite{DeAngelis&Ferrari18}, \cite{Ferrari&Riedel&Steg}, \cite{Guo&Tang&Xu18}, \cite{Guo&Xu18},   \cite{Kwon&Zhang15}, \cite{Steg}, \cite{Wang&Wang&Teo18}), although these problems have received an increasing interest in the recent years. We briefly discuss here some of these works. In \cite{Steg} it is determined a symmetric Nash equilibrium of a monotone-follower game with symmetric payoffs (i.e., the cost functional is the same for all players), and it is provided a characterization of any equilibria through a system of first order conditions. The same approach is also followed in \cite{Ferrari&Riedel&Steg} for a game in which  players are allowed to choose a regular control and a singular control. Such a problem has been motivated by a question arising in public economic theory. 
A general characterization of Nash equilibria through the Pontryagin Maximum Principle approach  has been investigated in the recent \cite{Wang&Wang&Teo18} for  regular-singular stochastic differential games. Connections between nonzero-sum games of singular control and games of optimal stopping have been tackled in \cite{DeAngelis&Ferrari18}. It is also worth mentioning some recent  works on mean field games with singular controls (see \cite{Fu&Horst17} and \cite{Guo&Lee17}) and their connection to symmetric $N$-player games (see \cite{Guo&Xu18}). A complete analysis of a Markovian $N$-player stochastic game in which players can control an underlying diffusive dynamic through a control of bounded-variation is provided in the recent \cite{Guo&Tang&Xu18}. There, the authors derive a Nash equilibrium by solving a system of \emph{moving} free boundary problems. General existence results for stochastic games with multi-dimensional singular controls and non-Markovian costs were, however, missing in the literature, and this has motivated our study.

\subsection{Our results} We now provide more details on our results by discussing the ideas and techniques of their proofs.
\smallbreak

\emph{The existence results.} Going back to the seminal ideas of J.\ Nash, a typical way to prove existence of Nash equilibria is to show existence of a fixed point for the best-reply map.  In the spirit of \cite{To}, our strategy to prove existence of Nash equilibria in the monotone-follower game and in the $n$-Lipschitz game is to exploit the submodular structure of our games in order to apply a lattice-theoretical fixed point theorem: the Tarski's fixed point theorem (see \cite{T}). We proceed as follows. We first endow the spaces of admissible strategies $\mathcal{A}$ and $\mathcal{L}(n)$ (defined above) with a lattice structure. While the lattice $\mathcal{L}(n)$ is complete, the same does not hold true for $\mathcal{A}$. To overcome this problem, we show that, under suitable assumptions, each ``reasonable''  strategy lives in a bounded subset of $\mathcal{A}$, and we restrict our analysis to this subset. 
We then prove that the best-reply maps are non empty. To accomplish this task in the $n$-Lipschitz game, we employ the so-called classical \emph{direct method}.
Indeed, since each strategy is forced to be $n$-Lipschitz, then the sequence of time-derivatives of any minimizing sequence is bounded in $\mathbb{L}^2$. 
Hence, Banach-Saks' theorem, together with the lower semi-continuity and the convexity of the costs, allows to conclude existence of the minima. On the other hand, for the monotone-follower game we use  some more recent techniques already employed to prove existence of optimizers in singular stochastic control problems (see \cite{Bank&Riedel01}). Assuming a uniform coercivity condition on the costs (which has to be, anyway, necessarily satisfied in any Nash equilibria; see Remark \ref{necessity.of.coercivity} below) we can use a theorem by Y.M.\ Kabanov (see Lemma 3.5 in \cite{K}) which gives relative sequential compactness, in the Ces\`aro sense, of any minimizing sequence. Then, exploiting again the lower semi-continuity and the convexity of the cost functions, we conclude existence of the minima. Next, we show that the best-reply maps preserve the order in the spaces of admissible strategies, and for this the submodular condition is essential. The existence then follows by invoking Tarski's fixed point theorem. 

Our result also generalizes to the monotone-follower game in which players are allowed to choose both a regular control and a singular control (see Remark \ref{existenceRegularSingular}). Moreover, some of our assumptions are not needed if we impose \emph{finite-fuel constraints} (see Remark \ref{FiniteFuelConstraint}).

It is worth stressing that our proof strongly hinges on the submodularity assumption, which, however, is a typical requirement in many problems arising in applications  (see, e.g., \cite{Milgrom&Roberts90}, \cite{To}, \cite{Vives90}, or the books \cite{Topkis11}  and \cite{Vives01} and the references therein).
\smallbreak
\emph{The approximation results.} 
Singular control problems naturally arise to overcome the ill-posedness of standard stochastic control problems in which the control linearly affects the dynamics of the state variable, and the cost of control is proportional to the effort.
Some kind of connection between regular control problems with the linear structure described above and singular control problems is then expected, and actually already discussed in the literature (see, e.g, the early \cite{Menaldi&Taksar89} for an analytical approach, and \cite{LZ} for a probabilistic approach). In Theorem 21 of \cite{LZ}, it is shown that any sequence obtained by choosing, for each $n\in \mathbb{N}$, a minimizer of the monotone-follower problem when the class of admissible controls is restricted to the set of $n$-Lipschitz controls, suitably approximates a (weak) optimal solution to the original monotone-follower problem.

We  prove that any sequence of Nash equilibria of the $n$-Lipschitz game is weakly relatively compact, and that any accumulation point is a weak Nash equilibrium of the monotone-follower game. We first show that this sequence satisfies a tightness criterion for the Meyer-Zheng topology. Then, we prove that  any Nash equilibrium of the $n$-Lipschitz game necessarily satisfies a system of  stochastic equations. After changing the underlying probability space by a Skorokhod representation, we  pass to the limit in these systems of equations and we deduce that any accumulation point solves  a new system of  stochastic equations. These equations can be viewed as a version of the Pontryagin maximum principle, and they are sufficient to ensure that the limit point is a Nash equilibrium in the new probability space, hence a weak Nash equilibrium. 

As a byproduct of this result, we are able to show that, for each $\eps >0$, there exists $n\in \mathbb{N}$ large enough such that the Nash equilibrium of the $n$-Lipschitz game is an $\eps$-Nash equilibrium of the monotone-follower game. This gives a clearer interpretation of the weak Nash equilibrium found through the approximation: the $N$-dimensional vector whose components are the expected costs associated to the weak Nash equilibrium is, in fact, a \emph{Nash equilibrium payoff} (as defined in \cite{Cardaliaguet2004}) of the monotone-follower game. 

\smallbreak
\emph{Applications and examples.}
Our existence result applies to deduce existence of open-loop Nash equilibria in stochastic differential games with  singular controls and non-Markovian random costs, whenever a certain structure is preserved by the dynamics. For the sake of illustration, we consider the case in which the dynamics of the state variable of each player are a linearly controlled geometric Brownian motion and a linearly controlled Ornstein-Uhlenbeck process.

Moreover, we consider the algorithm introduced by Topkis (see Algorithm II in \cite{To}) for submodular games: given as initial point the constantly null profile strategy, this algorithm consists of an iteration of the best-reply map. We show that, also in our setting, this algorithm converges to a Nash equilibrium.

\subsection{Organization of the paper} In Section \ref{sec:Definition of the Monotone-Follower Game}  we introduce the monotone-follower game.  Sections \ref{Section.Existence.Nash.Equilibria} and \ref{sec:nLipGame} are devoted to the existence theorems of Nash equilibria for the submodular monotone-follower game and for the $n$-Lipschitz game, respectively. The approximation results are contained in Section \ref{Section.Existence.Approximation.Weak.Nash}. The application of our result to suitable stochastic differential games is provided in Section \ref{Section.Application.Examples}, together with the proof of the convergence to a Nash equilibrium of a certain algorithm. In  Appendix \ref{App:MZ} we  recall some results about the Meyer-Zheng topology.

\subsection{Notations.}In the rest of this paper, for $m \in \mathbb{N}$ and $x,y \in \R^m$, we denote by $xy$ the scalar product in $\R^m$, as well as by $|\cdot|$ the Euclidean norm in $\R^m$. For $x,y \in \R^m$ and $c \in \R$,  we will write $x \leq y$ if $x^\ell \leq y^\ell$ for each $\ell=1,...,m$, as well as $x \leq c$ if $x^\ell \leq c $ for each $\ell=1,...,m$. Moreover, we set $x \land y := (x^1\land y^1,...,x^m\land y^m)$ and $x \lor y := (x^1 \lor y^1,...,x^m \lor y^m)$, where  $x^\ell \land y^\ell :=\min\{ x^\ell , y^\ell \}$ and $x^\ell \lor y^\ell :=\max\{ x^\ell , y^\ell \}$ for each $\ell=1,...,m$.
Finally, for $d, \, N \in \mathbb{N}$, and $a=(a^1,...,a^N) \in \R^{Nd}$, for each $i=1,...,N$ set  $a^{-i}:=(a^1,...,a^{i-1},a^{i+1},...,a^N) \in \R^{(N-1)d}$ and, for $v \in \R^d$, set $(v,a^{-i}):=(a^1,...,a^{i-1},v,a^{i+1},...,a^N)\in \R^{Nd}$.

\end{section}

%%%%%%%%%%%%%%%%%%%%%%%%%%%%%%%%%%%%%%%%%%%%%%%%%%%%%%%%%%%%%%%
%%%  GAMES OF MONOTONE FOLLOWER TYPE
%%%%%%%%%%%%%%%%%%%%%%%%%%%%%%%%%%%%%%%%%%%%%%%%%%%%%%%%%%%%%%%%

\begin{section}{The Monotone-Follower Game}
\label{Section.Games.of.Monotone.Follower.Type}
\subsection{Definition of the Monotone-Follower Game}
\label{sec:Definition of the Monotone-Follower Game} 

Fix a complete probability space $(\Omega, \mathcal{F}, \mathbb{P})$, a finite time horizon $T\in (0,\infty)$, an integer $N \geq 2$ and $k,d \in \mathbb{N}$. Consider a c\`adl\`ag process $L:\Omega \times [0,T] \rightarrow \mathbb{R}^k$, and, for $i=1,...,N$, assume to be given continuous semimartingales $f^i :\Omega \times [0,T] \rightarrow \mathbb{R}_+^d$, and set $f:=(f^1,...,f^N)$. Denote by $\bar{\mathbb{F}}_+^{f,L}=\{ \bar{\mathcal{F}}_{t+}^{f,L} \}_{t \in [0,T]}$ the right-continuous extension of the filtration generated by $f$ and $L$, augmented by the $\mathbb{P}$-null sets.

Define the space of \emph{admissible strategies}
\begin{equation}\label{def.admissible_strategy}
\mathcal{A}: = \left\{ \, V :\Omega \times [0,T] \rightarrow \R^d \, \bigg| \,\begin{matrix} V \text{ is an $\bar{\mathbb{F}}_+^{f,L}$-adapted  c\`adl\`ag process, with} \\ \text{ nondecreasing and nonnegative components} \end{matrix} \, \right\},
\end{equation}
and let $\mathcal{A}^N:= \bigotimes_{i=1}^N \mathcal{A}$ denote the set of \emph{admissible profile strategies}. In order to avoid confusion, in the following we will denote profile strategies in bold letters.
 
For each $i=1,...,N$, consider measurable functions $h^i,g^i:\mathbb{R}^k \times \mathbb{R}^{Nd} \rightarrow [0,\infty)$. We define the \emph{monotone-follower game} as the game in which each player $i \in \{ 1,...,N \}$ is allowed to choose an admissible strategy $A^i \in \mathcal{A}$ in order to minimize the cost functional 
\begin{align*}
\mathcal{J}^i(A^i,A^{-i} ):=& \mathbb{E} [C^i(f,L,\textbf{A})] := \mathbb{E} \bigg[ \int_0^T h^i(L_t, \textbf{A}_t)\, dt +g^i(L_T,\textbf{A}_T) + \int_{[0,T]} f_t^i \, dA_t^i   \bigg],
\end{align*}
where $A^{-i}:=(A^j)_{j\ne i}$ and $\textbf{A}:=(A^i,A^{-i}) \in \mathcal{A}^N$. Here and in the sequel the integrals with respect to $A^i$ are defined by
$$
\int_{[0,T]} f_t^i \, dA_t^i:= f_0^i A_0^i + \int_0^T f_t^i \, dA_t^i=  \sum_{\ell =1}^d f_0^{\ell,i} A_0^{\ell,i} + \sum_{\ell =1}^d \int_0^T f_t^{\ell} \, dA_t^{\ell,i}, 
$$
where the integrals on the right hand side are intended in the standard Lebesgue-Stieltjes sense on the interval $(0,T]$.

We recall the notion of Nash equilibrium.
\begin{definition}\label{def_of_Nash}
An admissible profile strategy \emph{$\bar{\textbf{A}} \in \mathcal{A}^N$} is a Nash equilibrium if, for every $i =1,...,N$, we have $\mathcal{J}^i(\bar{\mathbf{A}})< \infty$ and
$$
\mathcal{J}^i(\bar{A}^i,\bar{A}^{-i} ) \leq \mathcal{J}^i( V^i , \bar{A}^{-i} ), \quad \text{for every}\quad V^i \in \mathcal{A}.
$$   
\end{definition}
Letting $2^{\mathcal{A}} $ denote the set of all subset of $\mathcal{A}$, for each $i=1,...,N$ define the \emph{best-reply map} 
$R^i: \mathcal{A}^{N} \rightarrow 2^{\mathcal{A}} $ by 
\begin{equation}\label{bestreplymap}
R^i(\mathbf{A}):= \argmin_{V^i \in \mathcal{A}} \mathcal{J}^i(V^i,A^{-i}).
\end{equation}
Observe that the maps $R^i$ are constant in the variable $A^i$. Moreover define the map 
\begin{equation}\label{brm}
\textbf{R}:=(R^1,...,R^N):\mathcal{A}^N \rightarrow \bigotimes_{i=1}^N 2^{\mathcal{A}},
\end{equation}
and notice that the set of Nash equilibria coincides with the set of fixed points of the map $\textbf{R}$ which have finite values; that is, the set of $\bar{\mathbf{A}} \in \mathcal{A}^N$ such that  $\bar{\mathbf{A}} \in \mathbf{R}(\bar{\mathbf{A}})$ and $\mathcal{J}^i(\bar{\mathbf{A}})< \infty$ for every $i =1,...,N$.

\begin{remark}
The notion of equilibrium introduced above is that of the so-called Open-Loop Nash equilibrium. We focus on this specific class of equilibria since serious conceptual -- so far unsolved -- problems arise when one tries to define a game of singular controls with Closed-Loop strategies (see \cite{Back&Paulsen} for a discussion, and also \cite{Ferrari&Riedel&Steg} and \cite{Steg}).
\end{remark}

We now specify the structural hypothesis on the costs.
\begin{assumption}\label{H1} For each $i=1,...,N$ and for $\phi^i \in \{ h^i, g^i\}$ assume that:
\begin{enumerate}
\item \label{H1:1}  for each $(l,a^{-i}) \in \mathbb{R}^k\times\mathbb{R}^{(N-1)d}$, the function $\phi^i(l,\cdot,a^{-i})$ is lower semi-continuous, and strictly convex;
\item \label{H1:2} for each $l \in \mathbb{R}^k$ the function $\phi^i(l,\cdot,\cdot)$ has decreasing differences in $(a^i,a^{-i})$, i.e.
\begin{align*}
\phi^i(l,\bar{a}^i,{a}^{-i})- \phi^i(l,a^i,{a}^{-i})& \geq \phi^i(l,\bar{a}^i,\bar{a}^{-i})-  \phi^i(l,a^i,\bar{a}^{-i})  ,  
\end{align*}
for each $a,\bar{a} \in \mathbb{R}^{Nd}$ such that $\bar{a}\geq a$;
\item \label{H1:3} for each $(l,a^{-i}) \in \mathbb{R}^k\times\mathbb{R}^{(N-1)d}$, the function $\phi^i(l,\cdot,a^{-i})$ is submodular, i.e.
\begin{align*}
\phi^i(l,\bar{a}^i,{a}^{-i})+ \phi^i(l,a^i,{a}^{-i})& \geq \phi^i(l,\bar{a}^i \land a^i,{a}^{-i})+  \phi^i(l,\bar{a}^i \lor a^i,{a}^{-i})  ,  
\end{align*}
for each $a,\bar{a} \in \mathbb{R}^{Nd}$.
\end{enumerate}
\end{assumption}
Under Conditions $\ref{H1:2}$ and $\ref{H1:3}$ of Assumption \ref{H1} we refer to the game introduced above as to the \emph{submodular monotone-follower game} (see \cite{To} for a static deterministic $N$-player submodular game, and Theorem 2.6.1 and Corollary 2.6.1 at p.\ 44 in \cite{Topkis11} for further discussion on these conditions). The submodular structure of our game will play a fundamental role in the following.

%%%%%%%%%%%%%%%%%%%%%%%%%%%%%%%%%%%%%%%%%%%%%%%%%%%%%%%%%%%%%%%%%%
%%%               EXISTENCE RESULT
%%%%%%%%%%%%%%%%%%%%%%%%%%%%%%%%%%%%%%%%%%%%%%%%%%%%%%%%%%%%%%%

\subsection{Existence of Nash Equilibria in the Submodular Monotone-Follower Game}
\label{Section.Existence.Nash.Equilibria}

Define the space of \emph{extended admissible strategies}
\begin{equation}\label{def.ext.admissible_strategy}
\mathcal{A}_\infty := \left\{ \, V :\Omega \times [0,T] \rightarrow [0,\infty ]^d \, \bigg| \,\begin{matrix} V \text{ is an $\bar{\mathbb{F}}_+^{f,L}$-adapted  c\`adl\`ag process,} \\ \text{ with nondecreasing  components} \end{matrix} \, \right\},
\end{equation}
and, on it, we define the order relation $\preccurlyeq$ such that, for $V,U \in \mathcal{A}_\infty$, one has
\begin{equation*}
V \preccurlyeq U \quad \Longleftrightarrow \quad V_t \leq U_t \quad \forall \, t\in [0,T] , \quad \mathbb{P}\text{-a.s.}
\end{equation*}
Moreover, we can endow the space $\mathcal{A}_\infty$ with a lattice structure, defining the processes $V \land U$ and $V \lor U$ as
\begin{equation*}
(V \land U)_t:= V_t \land U_t \quad \text{and} \quad (V \lor U)_t:=V_t \lor U_t \quad \forall t \in [0,T],\quad \mathbb{P}\text{-a.s.}
\end{equation*} 
In the same way, on the set of \emph{extended profile strategies} $\mathcal{A}_\infty^N := \bigotimes_{i=1}^N \mathcal{A}_\infty $,  define, for $\mathbf{A},\mathbf{B} \in \mathcal{A}_\infty^N$, an order relation $\preccurlyeq^N$ by
\begin{equation*}
\mathbf{A} \preccurlyeq^N \mathbf{B} \quad \Longleftrightarrow \quad A^i \preccurlyeq B^i \quad \forall \, i \in \{1,...,N\},
\end{equation*}
together with the lattice structure
\begin{equation*}
\mathbf{A} \land \mathbf{B}:=(A^1 \land B^1,...,A^N \land B^N) \quad \text{and} \quad \mathbf{A} \lor \mathbf{B}:=(A^1 \lor B^1,...,A^N \lor B^N).
\end{equation*}

We now provide an existence result for the submodular monotone-follower game.  
\begin{theorem}\label{ExistenceNashEquilibrium}
Let Assumption $\ref{H1}$ hold and assume that the following uniform coercivity condition is satisfied: there exist two constants $K, \kappa>0$ such that, for each $i=1,...,N$, 
\begin{equation}\label{coercivity_condition}
\mathcal{J}^i(A^i,A^{-i})\geq \kappa\, \mathbb{E}[ |A^i_T| ]\quad \text{for all}\quad \mathbf{A} \in \mathcal{A}^N \quad \text{with} \quad \mathbb{E}[ |A^i_T| ] \geq K.   
\end{equation} 
Suppose, moreover, that there exists a constant $M>0$ such that, for each $i=1,...,N$,	
\begin{equation}\label{boundedness_of_h,g}
\text{for all} \quad \textbf{A} \in \mathcal{A}^N \quad \text{there exists} \quad r^i(\mathbf{A})\in \mathcal{A} \quad \text{such that} \quad   \mathcal{J}^i(r^i(\mathbf{A}),A^{-i}) \leq M .
\end{equation}
Then the set of Nash equilibria $F\subset \mathcal{A}^N$ is non empty, and the partially ordered set $(F,\preccurlyeq^N)$ is a complete lattice.
\end{theorem}
\begin{proof}
Our aim is to prove existence of a Nash equilibrium by applying Tarski's fixed point theorem (see Theorem 1 in \cite{T}) to the map $\mathbf{R}$ (cf.\ (\ref{brm})). For this, the assumption on the submodularity of $h^i$ and $g^i$ will play a crucial role. 

First of all, recalling $k,$ $K$ and $M$ from (\ref{coercivity_condition}) and (\ref{boundedness_of_h,g}), define the constant $ w:=\frac{2M}{\kappa} \lor K$, and introduce the set of restricted admissible strategies 
\begin{equation}\label{def.A(w)}
\mathcal{A}(w) := \{ A \in \mathcal{A} \, | \, \mathbb{E}[A_T^l] \leq w,\,\forall\, l=1,...,d\, \},
\end{equation} 
and the set of restricted profile strategies as $\mathcal{A}(w)^N := \bigotimes_{i=1}^N \mathcal{A}(w)$. 
In the following steps we will identify the proper framework allowing us to apply Tarski's fixed point theorem.

\bigbreak \noindent
\textit{(Step 1) The best-reply maps $R^i:\mathcal{A}^N \rightarrow \mathcal{A}(w) $ are well defined.}
\smallbreak \noindent 
Fix $i$ and take $\mathbf{A} \in \mathcal{A}^N$. We have to prove that there exists a unique $B \in \mathcal{A}$ such that 
$$
\mathcal{J}^i(B,A^{-i}) = \min\limits_{V \in \mathcal{A}} \mathcal{J}^i({V},A^{-i}),
$$
and, moreover, that $B \in \mathcal{A}(w)$. Clearly, by (\ref{bestreplymap}), we have ${B}=\{ {R}^i(\mathbf{A})_t \}_{t \in [0,T]}$.

Let $\{ V^j \}_{j \in \mathbb{N} } \subset \mathcal{A}$ be a minimizing sequence for the functional  $\mathcal{J}^i(\cdot,A^{-i})$.
Thanks to the coercivity conditions (\ref{coercivity_condition}) on the costs, we deduce that 
$$
\sup_{j \in \mathbb{N}} \mathbb{E}[| {V}_T^j|  ] < \infty.
$$
We can then use (a minimal adjustment of) Lemma 3.5 in \cite{K}, to find a c\`adl\`ag, nondecreasing, nonnegative, $\bar{\mathbb{F}}_+^{f,L}$-adapted process $B$, and a subsequence of $\{{V}^{j} \}_{j \in \mathbb{N} }$ (not relabeled) such that, $\mathbb{P}$-a.s., 
\begin{equation}\label{Komlos_limit}
\lim_m \int_{[0,T]} \varphi_t \, dB_t^m=\int_{[0,T]} \varphi_t \, dB_t  \quad \forall \, \varphi \in \mathcal{C}_b([0,T];\mathbb{R}^d) \quad \text{ and} \quad \lim_m B_T^m=B_T,
\end{equation}
where  we set, $\mathbb{P}$-a.s.
\begin{equation}
B_t^m := \frac{1}{m}	\sum_{j=1}^m {V}_t^j, \quad\quad \forall	t \in [0,T].
\end{equation}
Moreover, from the limit in (\ref{Komlos_limit}) we have that there exists a $\mathbb{P}$-null set $\mathcal{N}$ such that, for each $\omega \in \Omega \setminus \mathcal{N}$ there exists a subset $\mathcal{I}(\omega)\subset [0,T)$ of null Lebesgue measure, such that
\begin{equation*}
\lim_m B_t^m(\omega)=B_t(\omega) \quad \text{for each }\quad \omega \in \Omega \setminus \mathcal{N} \quad \text{and} \quad t \in [0,T]\setminus \mathcal{I}(\omega).
\end{equation*}
The latter convergence allows us to invoke Fatou's lemma which, together with the limit in (\ref{Komlos_limit}) and thanks to the lower semi-continuity of the costs, allows us to conclude that 
\begin{align*}
\mathcal{J}^i(B,A^{-i}) &\leq \liminf_m  \mathcal{J}^i(B^m,A^{-i}) \leq  \liminf_m \frac{1}{m} \sum_{j=1}^m \mathcal{J}^i(V^j,A^{-i}) = \min\limits_{V \in \mathcal{A}} \mathcal{J}^i({V},A^{-i}), 
\end{align*}
where we have used the convexity of $h^i$ and $g^i$ and the minimizing property of $V^j$. 
Hence $B$ is a minimizer for $\mathcal{J}^i( \cdot , A^{-i})$; in fact,  $B$ is the unique minimizer of $\mathcal{J}^i( \cdot , A^{-i})$ by strict convexity of the costs. 

It remains to prove that $B \in \mathcal{A}(w)$, and to accomplish that we argue by contradiction. 
If there exists $l\in\{1,...,d\}$ such that $\mathbb{E}[B_T^l] \geq w=\frac{2M}{\kappa} \lor K$, then we have $\mathbb{E}[ |B_T|] \geq \frac{2M}{\kappa} \lor K$ and hence, by the coercivity condition (\ref{coercivity_condition}) together with (\ref{boundedness_of_h,g}), we deduce that
$$
\mathcal{J}^i(B,A^{-i}) \geq \kappa \, \mathbb{E}[|B_T|] \geq 2M > \mathcal{J}^i(r^i(\mathbf{A}),A^{-i}), 
$$
which contradicts the optimality of $B$. 

\bigbreak \noindent
\textit{(Step 2) The best-reply maps $R^i$  are increasing, i.e.\  if $\mathbf{A}, \bar{\mathbf{A}} \in \mathcal{A}^N$ are such that $\mathbf{A} \preccurlyeq^N  \bar{\mathbf{A}}$, then $R^i(\mathbf{A}) \preccurlyeq R^i(\bar{\mathbf{A}})$.}
\smallbreak \noindent 
First of all, observe that, by an integration by parts (see, e.g., Corollary 2 at p.\ 68 in \cite{Protter05}), the cost functional rewrites as 
\begin{align}\label{integration-by-parts}
\mathcal{J}^i(A^i,A^{-i}) =\mathbb{E}  \bigg[ \int_0^T    h^i ( L_t,\mathbf{A}_t ) \, dt &+ g^i \left( L_T,\mathbf{A}_T \right)   -\int_0^T A_{t-}^i \, df_t^i + f_T^i {A}_T^i \bigg],
\end{align}
where $A_{t-}^i$ denotes the left-limit of $A_t^i$.
Thanks to the optimality of $R^i(\mathbf{A})$ we have the inequality
\begin{equation}\label{defofR}
\mathcal{J}^i(R^i(\bar{\mathbf{A}}) \land  R^i(\mathbf{A}) , A^{-i})-\mathcal{J}^i( R^i(\mathbf{A}), \mathbf{A}^{-i}) \geq 0,
\end{equation}
which by (\ref{integration-by-parts}) and setting $R^i:= R^i(\mathbf{A})$ and $\bar{R}^i:=R^i(\bar{\mathbf{A}})$, can be rewritten as 
\begin{align*}
\mathbb{E}  \left[ \int_0^T   \left( h^i(L_t, R_t^i \land \bar{R}_t^i,A_t^{-i})-h^i(L_t, R_t^i , A_t^{-i} ) \right) dt \right] -\mathbb{E} &\left[ \int_0^T  (R_{t-}^i \land \bar{R}_{t-}^i - {R}_{t-}^i)\,  df_t^i \right] \\
+ \mathbb{E}  \left[ g^i(L_T, R_T^i \land \bar{R}_T^i,A_T^{-i})-g^i(L_T, R_T^i , A_T^{-i} )  \right] &+\mathbb{E} \left[ f_T^i (R_T^i \land \bar{R}_T^i - {R}_T^i) \right] \geq 0,
\end{align*}
By the submodularity Condition \ref{H1:3} in Assumption \ref{H1}, we have
\begin{align}\label{submod.J.one}
\mathbb{E} & \left[ \int_0^T   \left( h^i(L_t, R_t^i \land \bar{R}_t^i,A_t^{-i})-h^i(L_t, R_t^i , A_t^{-i} ) \right)  dt \right] \\ \notag
& \leq \mathbb{E}  \left[ \int_0^T   \left( h^i(L_t,  \bar{R}_t^i, A_t^{-i})-h^i(L_t, R_t^i \lor \bar{R}_t^i , A_t^{-i} ) \right) dt \right],
\end{align}
and
\begin{align}\label{submod.J.two}
 \mathbb{E} & \left[ g^i(L_T, R_T^i \land \bar{R}_T^i,A_T^{-i})-g^i(L_T, R_T^i , A_T^{-i} )  \right] \\ \notag
& \leq \mathbb{E}  \left[ g^i(L_T,  \bar{R}_T^i,A_T^{-i})-g^i(L_T, R_T^i \lor \bar{R}_T^i , A_T^{-i} )  \right].
\end{align}
Moreover, one can easily verify that
\begin{equation}\label{submod.J.three}
\mathbb{E} \left[ \int_0^T  (R_{t-}^i \land \bar{R}_{t-}^i - {R}_{t-}^i)\, df_t^i \right] = \mathbb{E} \left[ \int_0^T  (\bar{R}_{t-}^i - R_{t-}^i \lor \bar{R}_{t-}^i )\, df_t^i \right]
\end{equation}
and
\begin{equation}\label{submod.J.four}
 \mathbb{E} \left[ f_T^i (R_T^i \land \bar{R}_T^i - {R}_T^i) \right] = \mathbb{E} \left[ f_T^i (\bar{R}_T^i  - R_T^i \lor \bar{R}_T^i) \right].
\end{equation}
Using (\ref{submod.J.one})-(\ref{submod.J.four}) we obtain 
\begin{align*}
& \mathcal{J}^i(R^i(\bar{\mathbf{A}}) \land  R^i(\mathbf{A}) , A^{-i})-\mathcal{J}^i( R^i(\mathbf{A}) , A^{-i}) \leq \mathcal{J}^i(R^i(\bar{\mathbf{A}}), A^{-i})-\mathcal{J}^i( R^i(\mathbf{A}) \vee R^i(\bar{\mathbf{A}}), A^{-i}),
\end{align*}
so that, by (\ref{defofR}), we deduce that
\begin{equation}\label{maggiorezero}
\mathcal{J}^i(R^i(\bar{\mathbf{A}}), A^{-i})-\mathcal{J}^i( R^i(\mathbf{A}) \vee R^i(\bar{\mathbf{A}}), A^{-i}) \geq 0.
\end{equation}
Now, by Condition \ref{H1:2} in Assumption \ref{H1}, we have
\begin{align*}
&\mathcal{J}^i(R^i(\bar{\mathbf{A}}), \bar{A}^{-i})-\mathcal{J}^i( R^i(\mathbf{A}) \vee R^i(\bar{\mathbf{A}}), \bar{A}^{-i})\geq \mathcal{J}^i(R^i(\bar{\mathbf{A}}), A^{-i})-\mathcal{J}^i( R^i(\mathbf{A}) \vee R^i(\bar{\mathbf{A}}), A^{-i}),
\end{align*}
and finally, by (\ref{maggiorezero}), we conclude that
\begin{align*}
\mathcal{J}^i(R^i(\bar{\mathbf{A}}), \bar{A}^{-i})-\mathcal{J}^i( R^i(\mathbf{A}) \vee R^i(\bar{\mathbf{A}}), \bar{A}^{-i}) \geq 0. 
\end{align*}
Hence $ R^i(\mathbf{A}) \vee R^i(\bar{\mathbf{A}}) $ minimizes $\mathcal{J}^i(\cdot, \bar{A}^{-i})$ as well as $R^i(\bar{\mathbf{A}})$ and, by uniqueness, it must be $ R^i(\mathbf{A}) \vee R^i(\bar{\mathbf{A}}) =R^i(\bar{\mathbf{A}})$. That is $R^i(\bar{\mathbf{A}}) \preccurlyeq R^i(\mathbf{A})$, which shows the claimed monotonicity.
\bigbreak
\noindent
\textit{(Step 3) The lattices $(\mathcal{A}_\infty^N,  \preccurlyeq^N )$ and $(\mathcal{A}_\infty, \preccurlyeq )$ are complete.} 
\smallbreak
\noindent
We prove the claim only for the lattice $(\mathcal{A}_\infty^N, \preccurlyeq^N )$, since an analogous rationale applies to show that the lattice  $(\mathcal{A}_\infty, \preccurlyeq )$ is complete.

To prove that the lattice $(\mathcal{A}_\infty^N, \preccurlyeq^N )$ is complete we have to show that each subset of $\mathcal{A}_\infty^N$ has a least upper bound and a greatest lower bound. We now prove only the existence of a least upper bound, since the existence of a greatest lower bound follows by similar arguments.

Consider a subset $\{ \mathbf{A}^j \}_{j \in \mathcal{I}} $ of $\mathcal{A}_\infty^N$, where $\mathcal{I}$ is a set of indexes. 
Define $Q:= ([0,T] \cap \mathbb{Q}) \cup \{ T \}$. For each $q \in Q$ we set 
\begin{equation}\label{def_of_the_sup_over_Q}
\tilde{ \mathbf{S} }_q := \esssup_{j\in\mathcal{I}} \mathbf{A}_q^j, 
\end{equation}
and we recall that there exists a countable subset $\mathcal{I}_q$ of $\mathcal{I}$ such that
\begin{equation}\label{sup_count}
\tilde{ \mathbf{S} }_q = \sup_{j \in \mathcal{I}_q } \mathbf{A}_q^j.
\end{equation} 
Define next the right-continuous process $\mathbf{S}:\Omega \times [0,T] \rightarrow [0,\infty]^{Nd}$ by 
\begin{equation}\label{least_upper_bound}
\mathbf{S}_T:= \tilde{\mathbf{S}}_T, \quad \text{and} \quad \mathbf{S}_t:= \inf\{ \, \tilde{\mathbf{S}}_q \,| \,   q > t, q \in Q \}, \quad \text{for} \quad t<T.
\end{equation}
Observe that $\mathbf{S}$ is  $\bar{\mathbb{F}}_+^{f,L}$-adapted by right-continuity of the filtration. Hence, $\mathbf{S}$ lies in $\mathcal{A}_\infty^N$, and clearly $\mathbf{A}^j \preccurlyeq^N \mathbf{S}$ for each $j\in \mathcal{I}$. 

Consider next an element $\mathbf{B}$ of $\mathcal{A}_\infty^N$ such that $\mathbf{A}^j \preccurlyeq^N \mathbf{B}$ for each $j\in \mathcal{I}$. For $q\in Q$ and $j \in \mathcal{I}_q$  there exists a $\mathbb{P}$-null set $\mathcal{M}_q^j$ such that $\mathbf{A}_q^j(\omega) \leq \mathbf{B}_q (\omega) $ for all $\omega \in \Omega \setminus \mathcal{M}_q^j$. Defining then $\mathcal{M}_q := \bigcup_{j \in \mathcal{I}_q } \mathcal{M}_q^j$, we have $\mathbf{A}_q^j(\omega) \leq \mathbf{B}_q (\omega) $ for all $\omega \in \Omega \setminus \mathcal{M}_q$ and $j\in \mathcal{I}_q$, which, by  (\ref{sup_count}), implies that $\tilde{\mathbf{S}}_q(\omega) \leq \mathbf{B}_q (\omega) $ for all $\omega \in \Omega \setminus \mathcal{M}_q$. Finally, introducing the $\mathbb{P}$-null set $\mathcal{M}:=\bigcup_{q \in Q } \mathcal{M}_q$, we have $\tilde{\mathbf{S}}_q(\omega) \leq \mathbf{B}_q (\omega) $ for all $\omega \in \Omega \setminus \mathcal{M}$ and $q \in Q$, and, by right-continuity, we deduce that $\mathbf{S} \preccurlyeq^N \mathbf{B}$. Thus, $\mathbf{S}$ is the least upper bound of $\{ \mathbf{A}^j \}_{j \in \mathcal{I}}$. 
\bigbreak\noindent
\textit{(Step 4) There exist increasing maps $\bar{R}^i: \mathcal{A}_\infty^N \rightarrow \mathcal{A}(w)$ such that $\bar{R}^i(\mathbf{A})= R^i(\mathbf{A})$ for each $\mathbf{A} \in \mathcal{A}^N$.} 
\smallbreak \noindent
For each $\mathbf{A} \in \mathcal{A}_\infty^N $, define $\bar{R}^i(\mathbf{A})$ as the least upper bound of the set $\{ R^i(\mathbf{V}) \, | \, \mathbf{V} \in \mathcal{A}^N , \mathbf{V} \preccurlyeq^N \mathbf{A} \}$ in the complete lattice $(\mathcal{A}_\infty, \preccurlyeq )$.
If $\mathbf{A} \in \mathcal{A}^N$, then $R^i(\mathbf{A}) \in \{ R^i(\mathbf{V}) \, | \, \mathbf{V} \in \mathcal{A}^N , \mathbf{V} \preccurlyeq^N \mathbf{A} \}$ and, since $R^i$ is increasing, $R^i(\mathbf{V}) \preccurlyeq R^i(\mathbf{A})$ for each $\mathbf{V} \in \mathcal{A}^N$ such that $\mathbf{V} \preccurlyeq^N \mathbf{A}$, which implies that $\bar{R}^i(\mathbf{A}) = R^i(\mathbf{A})$. 
Moreover, if $\mathbf{A}, \, \mathbf{B} \in \mathcal{A}_\infty^N$ are such that $\mathbf{A} \preccurlyeq^N \mathbf{B}$, then we have $\{ \mathbf{V} \in \mathcal{A}^N , \mathbf{V} \preccurlyeq^N \mathbf{A} \} \subset \{ \mathbf{V} \in \mathcal{A}^N , \mathbf{V} \preccurlyeq^N \mathbf{B} \}$ and hence  that $\bar{R}^i(\mathbf{A}) \preccurlyeq \bar{R}^i(\mathbf{B})$. It only remains to prove that $\bar{R}^i(\mathbf{A}) \in \mathcal{A}(w)$. In order to accomplish that, we observe that,  for each $\mathbf{V}, \, \mathbf{V}' \in \mathcal{A}^N$ such that $ \mathbf{V}, \, \mathbf{V}' \preccurlyeq^N \mathbf{A} $ we have that $\mathbf{V} \lor \mathbf{V}' \preccurlyeq^N \mathbf{A}$ and, since $R^i$ is increasing, $ R^i(\mathbf{V}) \lor R^i( \mathbf{V}')  \preccurlyeq R^i(\mathbf{V} \lor \mathbf{V}')$; that is, the set $\{ R^i(\mathbf{V}) \, | \, \mathbf{V} \in \mathcal{A}^N , \mathbf{V} \preccurlyeq^N \mathbf{A} \}$ is closed under taking maxima. This implies that there exists a sequence $\{ \mathbf{V}^j \}_{j \in \mathbb{N}} \subset \{ \mathbf{V} \in \mathcal{A}^N , \mathbf{V} \preccurlyeq^N \mathbf{A} \}  $ such that the sequence $\{ R^i(\mathbf{V}^j)_T \}_{j \in \mathbb{N}}$ is increasing and, moreover, 
\begin{equation}
\label{R.in.A(w)}
\bar{R}^i(\mathbf{A})_T = \lim_{j}R^i(\mathbf{V}^j)_T, \quad \mathbb{P}\text{-a.s.}, \quad \text{and} \quad \mathbb{E}[ \bar{R}^i(\mathbf{A})_T] = \lim_{j} \mathbb{E}[ R^i(\mathbf{V}^j)_T],  
\end{equation}
where the latter equality is due to the monotone convergence theorem.
Finally, by \emph{Step 1} we have that $R^i(\mathbf{V}^j) \in \mathcal{A}(w) $ for each $j\in \mathbb{N}$, which, by \eqref{R.in.A(w)}, implies that $\bar{R}^i(\mathbf{A}) \in \mathcal{A}(w).$

\bigbreak\noindent
\textit{(Step 5) Existence of Nash equilibria.}
\smallbreak \noindent
By the previous steps the lattice $(\mathcal{A}_\infty^N, \preccurlyeq^N ) $ is complete and the map $\bar{\mathbf{R}}:=(\bar{R}^1,...,\bar{R}^N)$ from the set of extended profile strategies $\mathcal{A}_\infty^N$ into itself is monotone increasing. Then, by Tarski's fixed point theorem (see \cite{T}, Theorem 1), the set of fixed point of the map $\bar{\mathbf{R}}$ is a non empty complete lattice. Now, by \emph{Step 4}, the image of the map $\bar{\mathbf{R}}$ is contained in $\mathcal{A}(w)^N$, and the map $\bar{\mathbf{R}}$ coincides with the map  $\mathbf{R}$ on $\mathcal{A}(w)^N$. This implies that the set of fixed points of $\mathbf{R}$ is equal to the set of fixed point of $\bar{\mathbf{R}}$, and  since such a set coincides with the set of Nash equilibria, the proof is completed.
\end{proof}

\subsection{Some Remarks} In this subsection we collect some remarks concerning  assumptions and  extensions of the previous theorem.
\begin{remark}[Comments on the Conditions of Theorem \ref{ExistenceNashEquilibrium}]\label{remark_on_conditions} A few comments are worth being done.
\begin{enumerate}
\item Condition (\ref{coercivity_condition}) is satisfied if, for example, there exists a constant $c>0$ such that 
\begin{equation*}
\mathbb{P} \left[ f_t^{i} \geq c, \, \forall i=1,...,N, \,\forall \, t \in [0,T] \right]=1,
\end{equation*}
or if $g^i$ are such that $g^i(l,a^i,a^{-i}) \geq \kappa \, | a^i|$. 
\item The role of Condition (\ref{boundedness_of_h,g}) is to force Nash equilibria, whenever they exist, to live in the bounded subset $\mathcal{A}^N(w)$ of $\mathcal{A}^N$. If there exist  measurable functions $H,G:\mathbb{R}^k \rightarrow [0,\infty)$ such that, for each $i=1,...,N$ and for each $(l,a^{-i}) \in \mathbb{R}^k\times\mathbb{R}^{(N-1)d}$, we have $h^i(l,0,a^{-i}) \leq H(l)$ and $g^i(l,0,a^{-i}) \leq G(l)$, with 
\begin{equation*}
\mathbb{E} \left[ \int_0^T  H(L_s) \, ds + G(L_T)\,\right] < \infty,
\end{equation*}
then  Condition (\ref{boundedness_of_h,g}) is satisfies with $r^i(\mathbf{A})=0$.
%\item example: $h(l, a^i,a^{-i})=\varphi(i,a)(1-\psi(a^{-i}))$ %with $\varphi,\psi$ increasing o decr.
\end{enumerate}
\end{remark}

\begin{remark}\label{necessity.of.coercivity} Consider the case $N=2,\, d=1$. The costs relative to Player 1 are $f^1=h^1=0,\, g^1(l,a^1,a^2)=e^{-a^1}(2-e^{-a^2})$, while the costs of Player 2 can be generic functions satisfying our requirements.  Then, all the assumptions of Theorem \ref{ExistenceNashEquilibrium} are satisfied, with the exception of the coercivity condition (\ref{coercivity_condition}), which is not satisfied by $\mathcal{J}^1$. If now $(\hat{A}^1,\hat{A}^2)$ were a Nash equilibrium, then for the first player we could write 
$$
0<\mathbb{E}[e^{-\hat{A}_T^1}(2-e^{-\hat{A}_T^2}) ] \leq \inf_{n \in \mathbb{N}} \mathbb{E}[e^{-n}(2-e^{-\hat{A}_T^2}) ] =0,
$$
which is clearly a contradiction. This example shows that, at least in the Nash equilibria, the coercivity condition (\ref{coercivity_condition}) is necessarily satisfied.
\end{remark}

\begin{remark}[Finite-Fuel Constraint]\label{FiniteFuelConstraint}
Many models in the literature on monotone-follower problems  enjoy a so-called \emph{finite fuel constraint} (see e.g.\ \cite{Karatzas85} for a seminal paper, and the more recent \cite{Bank05} and \cite{Chiarolla_Ferrari_Riedel}). This can be realized by requiring that the admissible control strategies stay bounded either $\mathbb{P}$-a.s.\ or in expectation. In our game, if we suppose that, for each $i=1,...,N$, the strategies of player $i$ belongs to the set 
$
\mathcal{A}(w^i):= \{ A \in \mathcal{A} \, | \, \mathbb{E}[A_T^l] \leq w^i, \, \forall \, l=1,...,d \, \},$
a proof similar to that of Theorem \ref{ExistenceNashEquilibrium} still shows existence of Nash equilibria without need of Conditions \ref{coercivity_condition} and \ref{boundedness_of_h,g}. 
\end{remark}
\begin{remark}[An Extension of Theorem \ref{ExistenceNashEquilibrium} with Regular-Singular Controls]\label{existenceRegularSingular} We here discuss how to extend Theorem \ref{ExistenceNashEquilibrium} to a game in which players can choose both a regular and a singular control.

Fix a square integrable random variable $\Theta$ and define the space of \emph{regular controls}
$ \mathcal{U}$ as the set of $\R^d$-valued  $\bar{\mathbb{F}}_+^{f,L}$-progressively measurable processes $u$ such that $|u_t| \leq \Theta \ \mathbb{P}\otimes dt-\text{a.e.}$ 
We consider the game of regular-singular controls, in which each player $i \in \{ 1,...,N \}$ is allowed to choose an admissible strategy $X^i=(u^i,A^i) \in \mathcal{U} \times \mathcal{A}$ in order to minimize the cost functional 
$$
\mathcal{J}^i(X^i,X^{-i} ):= \mathbb{E} \bigg[ \int_0^T h^i(L_t, {X}_t^1,...,X_t^N)\, dt +g^i(L_T,\textbf{A}_T) + \int_{[0,T]} f_t^i \, dA_t^i   \bigg].
$$
Define on $\mathcal{U}$ the order relation $\preccurlyeq$ by setting, for $u,v \in \mathcal{U}$, $u \preccurlyeq v$  if and only if $ u_t \leq v_t \ \mathbb{P}\otimes dt$-a.e. Next, consider on the lattice $(\mathcal{U},\preccurlyeq)$ the topology $\mathcal{I}$ of intervals (see, e.g., p.\ 250 in \cite{B}); that is, the topology for which the topology of closed sets is generated by the family of sets $\mathcal{I}_z:=\{ u \in \mathcal{U} \,:\, u \preccurlyeq z  \}$ and $\mathcal{I}^z:=\{ u\in \mathcal{U} \,:\, z \preccurlyeq u  \}$ for $z \in \mathcal{U}$. Since the topology $\mathcal{I}$ is included in the weak topology of $\mathbb{L}^2(\Omega\times [0,T];\mathbb{R}^{d})$ and $\mathcal{U}$ is bounded, then $\mathcal{U}$ is compact in the topology $\mathcal{I}$. Therefore, by a characterization of complete lattices (see Theorem 20 at p.\ 250 in \cite{B}), it follows that the lattice $(\mathcal{U},\preccurlyeq)$ is complete. Then, existence of Nash equilibria follows proceding as in the proof of Theorem \ref{ExistenceNashEquilibrium}.
\end{remark}

\end{section}

%%%%%%%%%%%%%%%%%%%%%%%%%%%%%%%%%%%%%%%%%%%%%%%%%%%%%%%%%%%%

\begin{section}{The $n$-Lipschitz Game}
\label{sec:nLipGame}

In the notation of Section \ref{Section.Games.of.Monotone.Follower.Type}, for each $n \in \mathbb{N}$, define the space of $n$-\emph{Lipschitz strategies}
\begin{equation*}
\mathcal{L}(n) = \left\{ A \in \mathcal{A} \, | \, A  \text{ is Lipschitz with Lipschitz constant smaller that }  n \text{ and } A_0=0 \right\},
\end{equation*}
and the space of $n$-\emph{Lipschitz  profile strategies}  as $\mathcal{L}^N(n)  := \bigotimes_{i=1}^N \mathcal{L}(n)$. The set $\mathcal{L}(n)$ (resp.\ $\mathcal{L}^N(n)$) inherits from $\mathcal{A}$ (resp.\ $\mathcal{A}^N$) the order relation $\preccurlyeq$ (resp.\ $\preccurlyeq^N$) together with the associated lattice structure.

For each $n \in \mathbb{N}$, the set of $n$-Lipschitz profile strategies  $\mathcal{L}^N(n)$, together with the cost functionals $\mathcal{J}^i$, define a game to which we will refer to as the \emph{$n$-Lipschitz game}. We say that an $n$-Lipschitz profile strategy $\mathbf{A} \in \mathcal{L}^N(n)$ is a Nash equilibrium of the $n$-Lipschitz game if, for each $i=1,...,N$, we have $\mathcal{J}^i(\mathbf{A}) < \infty$ and
$$
\mathcal{J}^i({A}^{i},{A}^{-i} ) \leq \mathcal{J}^i( V^i , {A}^{-i} ), \quad \text{for every} \quad V^i \in \mathcal{L}(n).
$$ 

\begin{theorem}[Existence of Nash Equilibria for the Submodular $n$-Lipschitz Game]\label{NashLipschitz}
Let Assumption \ref{H1} hold. Then, for each $n \in \mathbb{N}$, the set of Nash equilibria of the $n$-Lipschitz game $F\subset \mathcal{L}^N(n)$ is non empty, and the partially ordered set $(F,\preccurlyeq^N)$ is a complete lattice.
\end{theorem}

\begin{proof} As in the proof of Theorem \ref{ExistenceNashEquilibrium}, we identify the proper framework in order to apply Tarski's fixed point theorem. 
The completeness of the lattice $(\mathcal{L}^N(n),\preccurlyeq^N)$  follows by observing that the least upper bound (as well as the greatest lower bound) of any subset is still Lipschitz with Lipschitz constant bounded by $n$. 
Moreover, as in the proof of Proposition 26 at p.\ 109 in \cite{LZ}, we deduce that, 
for each  $i=1,...,N$ and each $\mathbf{A}\in \mathcal{L}^N(n)$, there exists a unique (by strict convexity of the costs) $R^i(\mathbf{A})\in \mathcal{L}(n)$ such that $$\mathcal{J}^i(R^i(\mathbf{A}),A^{-i})=\min_{V\in \mathcal{L}(n)} \mathcal{J}^i(V,A^{-i}).$$ 
By employing arguments as those in the \emph{Step 2} of the proof of Theorem \ref{ExistenceNashEquilibrium} we conclude that the map $\mathbf{R}=(R^1,...,R^N):\mathcal{L}^N(n) \rightarrow \mathcal{L}^N(n)$ is monotone increasing in the complete lattice $(\mathcal{L}^N(n), \preccurlyeq^N )$. Then, the thesis of the theorem follows from Tarski's fixed point theorem.
\end{proof}
\end{section}

%%%%%%%%%%%%%%%%%%%%%%%%%%%%%%%%%%%%%%%%%%%%%%%%%%%%%%%%%%
%    APPROXIMATION RESULT
%%%%%%%%%%%%%%%%%%%%%%%%%%%%%%%%%%%%%%%%%%%%%%%%%%%%%%%%%%%

\begin{section}{Existence and Approximation of Weak Nash Equilibria in the Submodular Monotone-Follower Game}
\label{Section.Existence.Approximation.Weak.Nash}

In this section we will investigate connections between the monotone-follower game and the $n$-Lipschitz games.

% WEAK FORMULATION %%%%%%%%%%%%%%%%%%%%%%%%%%%%%%%%%%%%%%%%%%%%%%%%%%%%%%%%%%%%%%%%%%%%%%%%%%%%%%%%%%%%%%%%%%%%%%%%%%

\subsection{Weak Formulation of the Monotone-Follower Game.}

For $T \in (0,\infty)$ and an arbitrary $m\in \mathbb{N}$, we introduce the following measurable spaces:
\begin{itemize}
 \item $\mathcal{C}_{+}^m$ denotes the set of $\mathbb{R}^m$-valued continuous function on $[0,T]$ with nonnegative components, endowed with the Borel $\sigma$-algebra generated by the uniform convergence norm;
 \item $\mathcal{D}^m$ denotes the Skorokhod space of $\mathbb{R}^m$-valued c\`adl\`ag functions, defined on $[0,T]$, endowed with the Borel $\sigma$-algebra generated by the Skorokhod topology;
 \item $\mathcal{D}_{\uparrow}^m$ denotes the Skorokhod space of $\mathbb{R}^m$-valued nondecreasing, nonnegative c\`adl\`ag functions, defined on $[0,T]$, endowed with the Borel $\sigma$-algebra generated by the Skorokhod topology.
 \end{itemize}
 Also, let $\mathcal{P}(\mathcal{C}_{+}^m)$, $\mathcal{P}( \mathcal{D}^m)$ and $\mathcal{P}(\mathcal{D}_{\uparrow}^m)$ denote the set of probability measures on the Borel $\sigma$-algebras of $\mathcal{C}_{+}^m$, $ \mathcal{D}^m$ and  $ \mathcal{D}_{\uparrow}^m$, respectively. Finally, denote by $\mathcal{P}(\mathcal{C}_{+}^m \times \mathcal{D}^m \times \mathcal{D}_{\uparrow}^m)$ the set of probability measures on the product $\sigma$-algebra.

Moreover, denote by $(\pi_f,\pi_L): \mathcal{C}_{+}^{Nd} \times \mathcal{D}^k \times [0,T] \rightarrow \R^{Nd+k}$ the canonical projection, i.e., set $(\pi_f,\pi_L)_t(f,L)=(f_t,L_t)$ for each $(f,L) \in \mathcal{C}_{+}^{Nd} \times \mathcal{D}^k$ and $t \in [0,T]$. Also, for a probability measure $\mathbb{P} \in \mathcal{P}(\mathcal{C}_{+}^{Nd} \times \mathcal{D}^{k})$, denote by $\bar{\mathbb{F}}_+^{\pi_f,\pi_L}$ the right continuous extension of the filtration on $\mathcal{C}_{+}^{Nd} \times \mathcal{D}^{k}$ generated by the canonical projections $\pi_f$ and $\pi_L$, augmented by the $\mathbb{P}$-null sets.

We now give a weak formulation of the monotone-follower game.  Assume to be given a distribution $\mathbb{P}_0 \in \mathcal{P}(\mathcal{C}_{+}^{Nd} \times \mathcal{D}^{k})$ such that the projection process $\pi_f: \mathcal{C}_{+}^{Nd} \times \mathcal{D}^{k} \times [0,T] \rightarrow \R^{Nd}$ is a semimartingale with respect to the filtration $\bar{\mathbb{F}}_+^{\pi_f,\pi_L}$.
\begin{definition}
We call a $basis$ a 5-tuple $\beta=(\Omega,\mathcal{F},\mathbb{P}, f,L)$ such that $(\Omega,\mathcal{F},\mathbb{P})$ is a complete probability space,  $L$ is an $\mathbb{R}^k$-valued c\`adl\`ag  process, $f=(f^1,...,f^N)$ is an $\mathbb{R}^{Nd}$-valued continuous, nonnegative  semimartingale with respect to the filtration $\bar{\mathbb{F}}_+^{f,L}$, and $\mathbb{P} \circ (f,L)^{-1}= \mathbb{P}_0$.
\end{definition}
For each basis $\beta$, we then give the relative notion of admissible strategy.
\begin{definition}
Given a basis $\beta=(\Omega,\mathcal{F},\mathbb{P}, f,L)$, an \emph{admissible strategy associated to $\beta$} is an $\mathbb{R}^d$-valued c\`adl\`ag, nondecreasing, nonnegative process on the probability space $(\Omega,\mathcal{F},\mathbb{P})$. We denote by $\mathcal{A}_\beta$ the set of admissible strategies associated to the basis $\beta$. Moreover, we define the space of admissible profile strategies associated to the basis $\beta$ as $\mathcal{A}_\beta^N:=\bigotimes_{i=1}^N \mathcal{A}_\beta.$ 
\end{definition}
Given a basis $\beta=(\Omega,\mathcal{F},\mathbb{P}, f,L)$, for each $i \in \{ 1,...,N \}$ and each admissible strategy $A^i \in \mathcal{A}_\beta$ we define the cost functionals
\begin{align*}
\mathcal{J}_\beta^i(A^i,A^{-i} )&:= \mathbb{E}^{\mathbb{P}} [C^i(f,L,\mathbf{A})] = \mathbb{E}^{\mathbb{P}} \bigg[ \int_0^T h^i(L_t, \mathbf{A}_t)\, dt + g^i(L_T,\mathbf{A}_T)  + \int_{[0,T]} f_t^i \, dA_t^i  \bigg],
\end{align*}
where $A^{-i}:=(A^j)_{j \ne i}$, $\mathbf{A}:=(A^i,A^{-i})$ and  $\mathbb{E}^{\mathbb{P}}$ denotes the expectation under the probability measure $\mathbb{P}$.

We finally introduce a notion of equilibrium that we will refer to as \emph{weak Nash equilibrium}.
\begin{definition}[Weak Nash Equilibrium]\label{def_weak_nash}
Given a basis $\bar{\beta}$ and an admissible profile strategy $\bar{\mathbf{A}} \in \mathcal{A}_{\bar{\beta}}^N$, we say that the couple $(\bar{\beta},\bar{\mathbf{A}})$ is a weak Nash equilibrium  if, for every $i=1,...,N$, we have
$$
\mathcal{J}_{\bar{\beta}}^i(\bar{A}^i,\bar{A}^{-i} ) \leq \mathcal{J}_{\bar{\beta}}^i( V^i , \bar{A}^{-i} ), \quad \text{for every} \quad V^i \in \mathcal{A}_{\bar{\beta}}.
$$   
\end{definition}

\subsection{Assumptions and a Preliminary Lemma}
In this subsection we specify the main assumptions of this section, we introduce some notations, and we provide a preliminary lemma.
\begin{assumption}\label{H2} Let Assumption \ref{H1} hold and, for each $i=1,...,N$, assume that:
\begin{enumerate}
\item \label{H2:1}  $g^i$ and $h^i$ are continuous and continuously differentiable in the variable $a^i \in \mathbb{R}^d$.
\item \label{H2:2} There exist $ \gamma_1, \gamma_2>1 $ such that the $d$-dimensional gradients $\nabla_i h^i$ and $\nabla_i g^i$ of the functions $h^i$ and $g^i$ with respect to the ($d$-dimensional) variable $a^i$ satisfy
\begin{equation}\label{H2:growth_h_and_g}
| \nabla_i h^i (l,a)|+|\nabla_i g^i (l,a)|   \leq C (1+|l|^{ \gamma_1}+|a|^{ \gamma_2}),
\end{equation}
for each $l \in \mathbb{R}^k$ and $a=(a^1,...,a^N) \in \mathbb{R}^{Nd}$.\\
Moreover, there exist measurable functions $H^i,G^i:\mathbb{R}^k \rightarrow \mathbb{R}$ such that $h^i(l,0,a^{-i}) \leq H^i(l)$ and $g^i(l,0,a^{-i}) \leq G^i(l)$, with 
\begin{equation}\label{H2:integrability}
\mathbb{E}^{\mathbb{P}_0} \left[ \int_0^T | H^i((\pi_{L})_s)|^q \, ds +|G^i((\pi_{L})_T)|^q \right] < \infty 
\end{equation}
and
\begin{equation}\label{H2:integrability.of.the.sup}
\mathbb{E}^{\mathbb{P}_0} \bigg[ \sup\limits_{s \in [0,T]} \left( | (\pi_{L})_s|^{\alpha \gamma_1 p}+| (\pi_{f})_s|^{\alpha  p} \,\right) \bigg] < \infty,
\end{equation}
where $q := \alpha \max \{\gamma_2\, p , \, p/(p-1) \}$ for some $p,\alpha>1$.
\item \label{H2:3} There exists a constant $c>0$ such that 
\begin{equation}\label{f>c}
\mathbb{P}_0 \left[ (\pi_{f})_t^{i} \geq c, \,\forall \, t \in [0,T], \, \forall \, i=1,...,N \right]=1,
\end{equation}
and the total conditional variation (see definition (\ref{conditionalvariation_T}) in the Appendix \ref{App:MZ}) of $\pi_L$ over the interval $[0,T]$ is finite; that is, $V_T^{\mathbb{P}_0}(\pi_L)<\infty$ .
\end{enumerate}
\end{assumption}
%\begin{remark}
%\textcolor{red}{example: $h(l, a^i,a^{-i})=\varphi(l,a)(1-%\psi(a^{-i}))$ with $\varphi,\psi$ increasing o decr.}
%\end{remark}
For a basis $\beta=(\Omega, \mathcal{F},  \mathbb{P}, f, L)$, a profile strategy  $\mathbf{A}=(A^1,...,A^N) \in \mathcal{A}_{\beta}^N$ and an index $i\in \{1,...,N\}$, we define the continuous (non adapted)  subgradient process $\partial C^i(f,L,\mathbf{A}) : \Omega \times [0,T] \rightarrow \R^d $ by setting
\begin{equation}\label{Y}
\partial C^i(f,L,\mathbf{A})_t:= \int_t^T \nabla_i  h^i (L_t,\mathbf{A}_t) \, dt +\nabla_i  g^i (L_T,\mathbf{A}_T)+ f_t^i, \quad  \forall \,t \in [0,T], \quad \mathbb{P}\text{-a.s.} 
\end{equation}
Furthermore, if  $\mathbf{A}$ is 
such that $\mathcal{J}_{\beta}^i(\mathbf{A}) < \infty$ for a certain $i\in\{1,...,N\}$, then, 
exploiting the convexity of $h^i$ and $g^i$  and integrating by parts, we obtain the following  subgradient inequality
\begin{equation}\label{erlemma0}
\mathcal{J}_{\beta}^i(B^i,{A}^{-i})-\mathcal{J}_{\beta}^i({A}^i,{A}^{-i}) \geq \mathbb{E}^{\mathbb{P}} \bigg[ \int_{[0,T]}  {\partial C}_t^i (dB_t^i-d{A}_t^i) \bigg], \quad \text{for each} \quad B^i \in \mathcal{A}_{\beta}.
\end{equation}

Fix a basis $\beta=(\Omega,\mathcal{F},\mathbb{P}, f,L)$  and recall that $\bar{\mathbb{F}}_+^{f,L}=\{ \bar{\mathcal{F}}_{t+}^{f,L} \}_{t \in [0,T]}$ is the right-continuous extension of the filtration generated by $f$ and $L$, augmented by the $\mathbb{P}$-null sets.  For each $n \in \mathbb{N}$, consider a Nash equilibrium $\mathbf{A}^n=(A^{1,n},...,A^{N,n})$ of the $n$-Lipschitz game as in Theorem \ref{NashLipschitz}.
The next lemma shows that any Nash equilibria of the $n$-Lipschtz game satisfy certain \emph{first order conditions}. The proof of this claim follows arguments analogus to those used in the proof of Proposition 27 in \cite{LZ}.  
\begin{lemma}\label{FOCn}
For every $n \in \mathbb{N}$ and every $i=1,..., N$, set  
$
\partial C^{i,n}:= \partial C^i(f,L,\mathbf{A}^n)$.
Then, under Assumption $\ref{H2}$, defining $\mathbf{1}:=(1,...,1)\in \mathbb{R}^d$, we have
\begin{equation}\label{FOCn1}
\mathbb{E}^{\mathbb{P}} \left[ \int_0^T \partial C_t^{i,n} d A_t^{i,n} \right] =-n \, \mathbb{E}^{\mathbb{P}} \left[ \int_0^T (\partial C_t^{i,n})^- \, \mathbf{1} \, dt \right] \quad \text{and} \quad \lim_{n } \mathbb{E}^{\mathbb{P}} \left[ \int_0^T (\partial C_t^{i,n})^-  \, dt \right] =0.
\end{equation}
\end{lemma}

\subsection{Existence and Approximation of Weak Nash Equilibria}
\label{subsection.Existence.Approximation.Weak.Nash.Equilibria}

We now state and prove the main result of this section, which can be thought of as a game-theoretic version of Theorem 21 in \cite{LZ}.

For an arbitrary $m \in \mathbb{N}$, consider on the space $\mathcal{C}_{+}^{m}$ the topology given by the convergence in the uniform norm. Furthermore, on the space $\mathcal{D}^m$ consider the \emph{pseudopath topology} $\tau_{pp}^{\text{\emph{\tiny T}}}$; that is, the topology on $\mathcal{D}^m$ induced by the convergence in the measure $dt + \delta_T$ on the interval $[0,T]$, where $dt$ denotes the Lebesgue measure, and $\delta_T$ denotes the Dirac measure at the terminal time $T$. The space $\mathcal{D}_{\uparrow}^{m}$ is a closed subset of the topological space $(\mathcal{D}^m ,\tau_{pp}^{\text{\emph{\tiny T}}})$, and the Borel $\sigma$-algebra induced by the topology $\tau_{pp}^{\text{\emph{\tiny T}}}$, coincides with the $\sigma$-algebra induced by the Skorokhod topology (see also the Appendix in \cite{LZ}).
Notice that the topological spaces $(\mathcal{D}^m ,\tau_{pp}^{\text{\emph{\tiny T}}})$ and $(\mathcal{D}_{\uparrow}^{m} ,\tau_{pp}^{\text{\emph{\tiny T}}})$ are separable, but not Polish (see, e.g., \cite{MZ}). Finally, on the product space $\mathcal{C}_{+}^{Nd} \times \mathcal{D}^k \times \mathcal{D}_{\uparrow}^{Nd}$, consider the product topology, and on $\mathcal{P}(\mathcal{C}_{+}^{Nd} \times \mathcal{D}^k \times \mathcal{D}_{\uparrow}^{Nd})$ consider the topology of weak convergence of probability measures.

Fix a basis $\beta=(\Omega,\mathcal{F},\mathbb{P}, f,L)$  and consider, for each $n\in \mathbb{N}$, a Nash equilibrium $\mathbf{A}^n=(A^{1,n},...,A^{N,n})$ of the $n$-Lipschitz game as in Theorem \ref{NashLipschitz}.
Define, for $n \in \mathbb{N}$, the law  $\mathbb{P}^n:= \mathbb{P} \circ (f,{L},{\mathbf{A}}^n)^{-1} $ in $\mathcal{P}(\mathcal{C}_{+}^{Nd} \times \mathcal{D}^k \times \mathcal{D}_{\uparrow}^{Nd})$; with a slight abuse of terminology, we will refer to the law $\mathbb{P}^n$ as the law of the Nash equilibrium $\mathbf{A}^n$. We then have the following theorem.
 
\begin{theorem}\label{th}
Under Assumption \ref{H2} the following statements hold. 
\begin{enumerate}
\item  The sequence $\{ \mathbb{P}^n \}_{n \in \mathbb{N} }$ of the laws of the Nash equilibria of the $n$-Lipschitz games is weakly relatively compact in $\mathcal{P}(\mathcal{C}_{+}^{Nd} \times \mathcal{D}^k \times \mathcal{D}_{\uparrow}^{Nd})$. 
\item Any accumulation point $\bar{\mathbb{	P}}$ is the law of a weak Nash equilibrium of the monotone-follower game; that is, there exist a basis $\bar{\beta}=(\bar{\Omega}, \bar{\mathcal{F}}, \bar{\mathbb{Q}},\bar{f},\bar{L})$ and an admissible profile strategy $\bar{\mathbf{A}} \in \mathcal{A}_{\bar{\beta}}^N$, such that $(\bar{\beta},\bar{\mathbf{A}})$ is a weak Nash equilibrium of the monotone-follower game and $\bar{\mathbb{P}}=\bar{\mathbb{Q}} \circ (\bar{f},\bar{L},\bar{\mathbf{A}})^{-1}$.  
\end{enumerate}
\end{theorem}
\begin{proof} We prove the two claims of the theorem separately. 
\bigbreak \noindent
\emph{Proof of Claim 1}.
By assumption we  have $V_T^{\mathbb{P}}(L) < \infty$. Moreover, by employing arguments similar to those in the proof of Proposition 28 at p.\ 110 in \cite{LZ}, we find
\begin{equation}\label{stimafondamental}
\sup\limits_{n  } \mathbb{E}^{\mathbb{P}} \left[ | \mathbf{A}_T^{n} |^q \right] < \infty,
\end{equation}
where $q>1$ is as in Assumption \ref{H2}. Therefore, from Lemma \ref{MeyerZheng}, we can deduce that the sequence   $\{ {\mathbf{A}}^n \}_{n\in \mathbb{N}}$ is tight in $\mathcal{P}(\mathcal{D}_{\uparrow}^{N d})$, and that $L$ in tight in $\mathcal{P}(\mathcal{D}^k)$. 
Furthermore, since the space $\mathcal{C}_{+}^{Nd}$ is Polish, $\mathbb{P} \circ f^{-1}$ is regular, and hence $f$ is tight in $\mathcal{P}(\mathcal{C}_{+}^{Nd})$ (see, e.g., Remark 13.27 at p.\ 260 in \cite{Klenke13}). This implies that the sequence $\{ (f,{L},{\mathbf{A}}^n) \}_{n\in \mathbb{N}}$ is tight in $\mathcal{P}(\mathcal{C}_+^{Nd} \times \mathcal{D}^{k} \times \mathcal{D}_{\uparrow}^{N d})$.

By Prokhorov's theorem (see, e.g., Theorem 13.29 at p.\ 261 in \cite{Klenke13}), there exists a subsequence of indexes (still denoted by $n$) and a probability measure $\bar{\mathbb{P}}\in \mathcal{P}(	\mathcal{C}_+^{Nd} \times \mathcal{D}^{k} \times \mathcal{D}_{\uparrow}^{N d} )$ such that the sequence ${\mathbb{P}}^n$ converges weakly to $\bar{\mathbb{P}}$. The first claim of the theorem is thus proved.

\bigbreak\noindent
\emph{Proof of Claim 2}.
Thanks to an extension of Skorokhod's theorem for separable spaces (see Theorem 3 in \cite{D}), there exists a probability space $(\bar{\Omega},\bar{\mathcal{F}}, \bar{\mathbb{Q}})$, and, on it, a sequence
$
\{ (\bar{f}^n, \bar{L}^n,\bar{\mathbf{A}}^n)\}_{n\in \mathbb{N}}
$
of $\mathcal{C}_+^{Nd}\times \mathcal{D}^{k} \times \mathcal{D}_{\uparrow}^{N d} $-valued random variables, and a $\mathcal{C}_+^{Nd} \times \mathcal{D}^{k} \times \mathcal{D}_{\uparrow}^{N d}$-valued random variable $(\bar{f},\bar{L},\bar{\mathbf{A}})$, such that  $\bar{\mathbb{Q}} \circ (\bar{f}^n, \bar{L}^n,\bar{\mathbf{A}}^n)^{-1}={\mathbb{P}}^n$ and  $\bar{\mathbb{Q}} \circ(\bar{f}, \bar{L},\bar{\mathbf{A}} )^{-1}=\bar{\mathbb{P}}$. 
Furthermore, this representation is such that, for almost all $\omega \in \bar{\Omega}$, we have
\begin{equation}\label{convunifoff}
\bar{f}^n(\omega)\rightarrow \bar{f}(\omega) \quad \text{uniformly on the interval } [0,T],
\end{equation}
as well as
\begin{equation}\label{convinmeasurelamda}
( \bar{L}^n(\omega),\bar{\mathbf{A}}^n(\omega) ) \rightarrow ( \bar{L}(\omega),\bar{\mathbf{A}}(\omega) ) \quad \text{ in the measure $dt+\delta_T$ on $[0,T]$}.
\end{equation}
Define then $\bar{\beta}:= (\bar{\Omega},\bar{\mathcal{F}}, \bar{\mathbb{Q}},\bar{f},\bar{L})$. Since $\mathbb{P}\circ(f,{L})^{-1}$ is constantly $\mathbb{P}_0$, then the same holds for its limit; that is, $\bar{\mathbb{Q}} \circ(\bar{f}, \bar{L})^{-1}=\mathbb{P}_0$, and this implies that $\bar{\beta}$ is a basis.

Next, for every $i=1,....,N$ and $n\in \mathbb{N}$, recalling (\ref{Y}), we define on the probability space $(\bar{\Omega},\bar{\mathcal{F}}, \bar{\mathbb{Q}})$ the subgradient processes $\partial \bar{C}^{i,n}:=\partial C^i(\bar{f}^n, \bar{L}^n,\bar{\mathbf{A}}^n)$ and $\partial \bar{C}^{i}:=\partial C^i(\bar{f}, \bar{L},\bar{\mathbf{A}})$.
 By the convergence at the terminal time (\ref{convinmeasurelamda}) together with Fatou's lemma and the estimate $(\ref{stimafondamental})$ we have
\begin{equation}\label{stima.A}
\mathbb{E}^{\bar{\mathbb{Q}}} [|\bar{\mathbf{A}}_T|^q] \leq \sup_n \mathbb{E}^{\bar{\mathbb{Q}}}[|\bar{\mathbf{A}}_T^n|^q] = \sup_n \mathbb{E}^{{\mathbb{P}}}[|\mathbf{A}_T^n|^q] < \infty.
\end{equation}
Let $ Q:=([0, T) \cap \mathbb{Q}) \cup \{T\}$ and define the measurable function $\Phi: \mathcal{D}^k \rightarrow \mathbb{R}$ by 
$$
\Phi(X):= \sup\limits_{ t \in Q  } |X_t|. $$
Being constantly equal to $\mathbb{P} \circ \Phi(L)^{-1}$, the sequence $\{ \bar{\mathbb{Q}} \circ \Phi(\bar{L}^n)^{-1} \}_{n\in \mathbb{N}}$ is tight in $\mathcal{P}(\R^k)$. This allows to assume without loss of generality (modulo a further subsequence, a new Skorokhod representation of the sequence $\{ (\bar{f}^n,\Phi(\bar{L}^n), \bar{L}^n, \bar{\mathbf{A}}^n)\}_{n\in \mathbb{N}} $, and exploiting the measurability of $\Phi$), that $\Phi(\bar{L}^n)$ converges to $\Phi(\bar{L})$, $\mathbb{Q}$-a.s. Furthermore, by (\ref{H2:integrability.of.the.sup}) in Assumption \ref{H2}, we have
$
\mathbb{E}^{\bar{\mathbb{Q}}} [\Phi(\bar{L})] = \mathbb{E}^{ {\mathbb{P}}_0} [\Phi({\pi_L})]  <\infty.
$
The latter, together with the $\bar{\mathbb{Q}}$-a.s.\ convergence of $\Phi(\bar{L}^n)$, the convergence in (\ref{convinmeasurelamda}), and the integrability proved in (\ref{stima.A}), implies that, for $\bar{\mathbb{Q}}$-almost all $\omega \in \bar{\Omega}$, there exists a constant $M(\omega)<\infty$ such that
$$
\sup_n \sup_{t \in [0,T]} (|\bar{L}_t^n(\omega)| + |\bar{\mathbf{A}}_t^n(\omega)| +|\bar{L}_t(\omega)| + |\bar{\mathbf{A}}_t(\omega)| ) \leq M(\omega).  
$$
Thus, for $\bar{\mathbb{Q}}$-almost all $\omega \in \Omega$, we can find, by continuity of $h^i$, another constant $K(\omega)<\infty$ such that 
$$
\sup_n \sup_{t \in [0,T]}  \left[ h^i(\bar{L}_t^n(\omega),\bar{B}_t^n(\omega),  \bar{A}_t^{-i,n}(\omega) )+ h^i(\bar{L}_t(\omega), \bar{B}_t(\omega), \bar{A}_t^{-i}(\omega) )\right] \leq K(\omega).
$$
Hence, for $\bar{\mathbb{Q}}$-almost all $\omega \in \Omega$, the bounded continuous function $\nabla_i h^i (l,a)\land K(\omega)$ coincides with the function $\nabla_i h^i (l,a) $ when evaluated along the sequence $( \bar{L}_s^n(\omega),\bar{\mathbf{A}}_s^n(\omega))$ and at the limit point $( \bar{L}_s(\omega),\bar{\mathbf{A}}_s(\omega))$. 

Considering $\omega$ fixed and $\nabla_i h^i$ bounded by $K(\omega)$, this allows to use equation (\ref{charatconv}),  together with standard arguments exploiting the compactness of $[0,T]$, in order to deduce that,  $\bar{\mathbb{Q}}$-a.s.
\begin{equation}\label{unifconvY}
\lim_n \sup\limits_{t \in [0,T]} \left| \int_t^T \left( \nabla_i h^i(\bar{L}_s^n,\bar{\mathbf{A}}_s^n) - \nabla_i h^i(\bar{L}_s,\bar{\mathbf{A}}_s) \right)  ds \, \right|=0. 
\end{equation}
The latter, thanks to  (\ref{convunifoff}) and (\ref{convinmeasurelamda}) and to the continuity of $\nabla_i g^i$,  implies that, 
\begin{equation}\label{uniformconvergengeofYn} 
  \partial \bar{ C}^{i,n} \rightarrow \partial \bar{ C}^i \quad \text{uniformly on the interval } [0,T], \quad \text{for every } i=1,...,N, \quad \bar{\mathbb{Q}}\text{-a.s.}
\end{equation} 
The following claims summarize two key properties of the processes $\partial \bar{ C}^i$ and $\bar{\mathbf{A}}$ that will guarantee that $(\bar{\beta},\bar{\mathbf{A}})$ is a weak Nash equilibrium as in Definition  \ref{def_weak_nash}.
\smallbreak \noindent
For every $i=1,...,N$, we now prove that the following hold $\bar{\mathbb{Q}}$-a.s.:
\begin{enumerate}
\item[(2.a)]\label{2.a} $\partial \bar{ C}_t^{i}\geq 0$ for every $t \in [0,T]$;
\item[(2.b)]\label{2.b} $\displaystyle \int_{[0,T]} \partial \bar{ C}_t^i \, d \bar{A}_t^i=0$.
\end{enumerate} 
\noindent
\textit{(Proof of 2.a)}
We begin by proving that $\partial \bar{ C}^n \rightarrow \partial \bar{ C}$ in $\mathbb{L}^1( \bar{\mathbb{Q}} \otimes dt)$. For $i=1,...,N$, from the convergence proved in (\ref{uniformconvergengeofYn}) we have that $  \bar{\mathbb{Q}} \otimes dt$-a.e. $\partial \bar{ C}^{i,n}$ converges  to $\partial \bar{ C}^i$. Moreover, for $p>1$ as in  Assumption \ref{H2},  by the growth condition (\ref{H2:growth_h_and_g}) we easily find that 
\begin{align}\label{uniformintegrability}
\mathbb{E}^{\bar{\mathbb{Q}}} \bigg[ \sup_{ t \in [0,T]} |\partial \bar{ C}_t^{i,n}|^p \bigg] &\leq  \widetilde{C} \, \bigg(  1 + \mathbb{E}^{{\mathbb{P}}} [ |{\mathbf{A}}_T^n|^{\gamma_2 p}] + \mathbb{E}^{\mathbb{P}_0} \bigg[ \sup_{ t \in [0,T]} \left(   | {(\pi_L)}_t |^{\gamma_1 p}  + | {(\pi_f)}_t^{i} |^p \right) \bigg] \bigg),
\end{align}
for a suitable constant $\widetilde{C}$. Using then the integrability condition (\ref{H2:integrability.of.the.sup}) in Assumption \ref{H2} and the estimates (\ref{stimafondamental}) (recall that by assumption $\gamma_2 p<q$), we have 
\begin{equation}\label{uniformintY}
\sup_{n } \mathbb{E}^{\bar{\mathbb{Q}}} \bigg[ \sup_{ t \in [0,T]} |\partial \bar{ C}_t^{i,n}|^p \bigg] <\infty,
\end{equation}
which implies that the sequence $\partial \bar{ C}^{i,n}$ is uniformly integrable. From Theorem 6.25 at p. in \cite{Klenke13}, we deduce then that $\partial \bar{ C}^n \rightarrow \partial \bar{ C}$ in $\mathbb{L}^1( \bar{\mathbb{Q}} \otimes dt)$.
Now, from the second equation in (\ref{FOCn1}) in Lemma \ref{FOCn}, we find
\begin{equation*}
0=\lim_{n} \mathbb{E}^{\mathbb{P}} \left[ \int_0^T ({\partial C}_t^{i,n})^- \, dt \, \right]=\lim_{n} \mathbb{E}^{\bar{\mathbb{Q}}} \left[ \int_0^T (\partial \bar{ C}_t^{i,n})^- \, dt \, \right]=\mathbb{E}^{\bar{\mathbb{Q}}} \left[ \int_0^T (\partial \bar{ C}_t^{i})^- \, dt \, \right],
\end{equation*}
and by continuity of $\partial \bar{ C}^{i}$ we conclude that $\bar{\mathbb{Q}}$-a.s.
\begin{equation}\label{first}
\partial \bar{ C}_t^{i}\geq 0  	\, ,\ \forall \, t \in [0,T], \, \, \forall \, i=1,...,N.
\end{equation}
\vspace{0.15 cm}\noindent
\textit{(Proof of 2.b)}
Computations analogous to those employed in (\ref{uniformintegrability}) yield
\begin{equation}\label{uniformintegrability_Y^n.alpha_p}
\mathbb{E}^{\bar{\mathbb{Q}}}\bigg[ \sup_{ t \in [0,T] } |\partial \bar{ C}_t^{i,n}|^{\alpha p} \bigg] 
\leq \widetilde{C} \, \bigg( 1+ \mathbb{E}^{{\mathbb{P}}}[|{\mathbf{A}}_T^n|^{\alpha \gamma_2 p}]+\mathbb{E}^{\mathbb{P}_0} \bigg[  \sup_{ t \in [0,T] } \left(   | {(\pi_L)}_t |^{\alpha \gamma_1 p}  + | {(\pi_f)}_t^{i} |^{\alpha p} \right)\bigg] \bigg)	,
\end{equation} 
as well as,
\begin{equation}\label{integrability_Y^alpha_p}
\mathbb{E}^{\bar{\mathbb{Q}}}\bigg[ \sup_{ t \in [0,T]} |\partial \bar{ C}_t^{i}|^{\alpha p} \bigg] 
\leq \widetilde{C} \, \bigg( 1+ \mathbb{E}^{\bar{\mathbb{Q}}}[|\bar{\mathbf{A}}_T|^{\alpha \gamma_2 p}]+\mathbb{E}^{\mathbb{P}_0} \bigg[  \sup_{ t \in [0,T]} \left(   | {(\pi_L)}_t |^{\alpha \gamma_1 p}  + | {(\pi_f)}_t^{i} |^{\alpha p} \right)\bigg] \bigg).
\end{equation}
Now, the estimates (\ref{stimafondamental}), (\ref{stima.A}), (\ref{uniformintegrability_Y^n.alpha_p}) and (\ref{integrability_Y^alpha_p}) imply that
$$
\sup_n \mathbb{E}^{\bar{\mathbb{Q}}}\left[ \sup_{ t \in [0,T] } |\partial \bar{ C}_t^{i,n}|^{\alpha p} +\sup_{ t \in [0,T]} |\partial \bar{ C}_t^{i}|^{\alpha p} + |\bar{\mathbf{A}}_T^n|^{\frac{\alpha p }{p-1}} + |\bar{\mathbf{A}}_T|^{\frac{\alpha p }{p-1}} \right] < \infty, 
$$
which, together with the convergence established in (\ref{uniformconvergengeofYn}), allows us to use Lemma \ref{existenceofasubsequenceforlimitofcosts} in Appendix  \ref{App:MZ} in order to deduce that 
\begin{equation}\label{limit_Y^idA^i_in_proof}
\mathbb{E}^{\bar{\mathbb{Q}}} \bigg[ \int_{[0,T]} \partial \bar{ C}_t^i \, d \bar{A}_t^i \, \bigg] = \lim_n \mathbb{E}^{\bar{\mathbb{Q}}} \bigg[ \int_{[0,T]} \partial \bar{ C}_t^{i,n} \, d \bar{A}_t^{i,n} \, \bigg] = \lim_n \mathbb{E}^{\mathbb{P}} \bigg[ \int_0^T {\partial C}_t^{i,n} \, d{A}_t^{i,n} \, \bigg] \leq 0 ,
\end{equation}
where we have used the first equality of (\ref{FOCn1}) in Lemma \ref{FOCn} and that, for each $n \in \mathbb{N}$,  $\bar{A}_0^{i,n}=0$ $\bar{\mathbb{Q}}$-a.s. 
This implies, thanks to the non negativity of $\partial \bar{ C}^i$ established in (\ref{first}), that $\bar{\mathbb{Q}}$-a.s.
$$
\int_{[0,T]} \partial \bar{ C}_t^i d \bar{A}_t^i=0;
$$
i.e. (\emph{2.b}) is proved.
\smallbreak 
It does remain to conclude that the couple $(\bar{\beta},\bar{\mathbf{A}})$ is a weak Nash equilibrium of the game.
Fix $i\in\{1,...,N\}$, and consider an admissible strategy $B^i \in \mathcal{A}_{\bar{\beta}}$. By (\ref{erlemma0}) and Claims (\emph{2.a}) and (\emph{2.b}) we have
\begin{align*}
& \mathcal{J}_{\bar{\beta}}^i(B^i,\bar{A}^{-i})-\mathcal{J}_{\bar{\beta}}^i(\bar{A}^i,\bar{A}^{-i}) \geq \mathbb{E}^{\bar{\mathbb{Q}}} \bigg[ \int_{[0,T]}  \partial \bar{ C}_t^i (dB_t^i-d\bar{A}_t^i) \bigg]= \mathbb{E}^{\bar{\mathbb{Q}}} \bigg[ \int_{[0,T]}  \partial \bar{ C}_t^i dB_t^i \bigg]\geq 0,
\end{align*}
which in fact completes the proof.
\end{proof}

%%%%%%%%%%%%%%%%%%%%%%%%%%%%%%%%%%%%%%%%%%%%%%%%%%%%%%%%%%%%%%%%%%%%%%%%%%%%%

\subsection{On Lipschitz  $\varepsilon$-Nash Equilibria for the Monotone-Follower Game}
\label{sec:LipepsilonEq}

In this subsection we prove another connection between the Lipschitz games and the monotone-follower game by showing that $\varepsilon$-Nash equilibria of the monotone-follower game can be realized as Nash equilibria of the $n$-Lipschitz game, for $n$ sufficiently large. The proof of this result exploits Theorem \ref{th}, combined with a contradiction scheme.

As in Subsection \ref{subsection.Existence.Approximation.Weak.Nash.Equilibria}, in the following we consider fixed a basis $\beta=(\Omega, \mathcal{F},  \mathbb{P}, f, L)$, and,  for each $n \in \mathbb{N}$, let $\mathbf{A}^n=(A^{1,n},...,A^{N,n})$ be a Nash equilibrium of the $n$-Lipschitz game as in Theorem \ref{NashLipschitz}.

\begin{theorem}
\label{Existence_of_epsilon}
Suppose that Assumption \ref{H2} holds and that there exists a constant $C>0$ such that
\begin{equation}\label{growth_h_and_g_epsilon_Nash}
| h^i (l,a)|+|g^i (l,a)|   \leq C(1+|l|^{\gamma_1} + |a^{-i}|^{\gamma_2} ),
\end{equation}
for each $l \in \mathbb{R}^k$ and $a=(a^1,...,a^N) \in \mathbb{R}^{Nd}$.

Then, for each $\varepsilon>0$, there exists $n_\varepsilon$ such that the Nash equilibrium $\mathbf{A}^{n_\varepsilon}$ of the $n_\varepsilon$-Lipschitz game is an $\varepsilon$-Nash equilibrium of the monotone-follower game; that is, for each $i=1,...,N$
$$
\mathcal{J}_\beta^i(A^{i,n_\varepsilon},A^{-i,n_\varepsilon}) \leq \mathcal{J}_\beta^i(B^i,A^{-i,n_\varepsilon}) + \varepsilon \quad \text{for each} \quad B^i \in \mathcal{A}_{\beta}.
$$
\end{theorem}
\begin{proof}
We argue by contraddiction and we suppose that the thesis is false. Then, there exists $\varepsilon>0$ such that, for each $n \in \mathbb{N}$, there exist $i_n\in\{1,...,N\}$ and an admissible strategy $B^n \in \mathcal{A}_{\beta}$ with 
$$
\mathcal{J}_\beta^{i_n}(\mathbf{A}^n) > \mathcal{J}_\beta^{i_n}(B^n,A^{-i_n,n}) + \varepsilon.
$$
Since the number of indexes of the players is finite,
we can suppose that there exists $i\in \{1,...,N\}$ such that, for each $n \in \mathbb{N}$, 
\begin{equation}\label{equazione_assurda!!!}
\mathcal{J}_\beta^{i}(\mathbf{A}^n) > \mathcal{J}_\beta^{i}(B^n,A^{-i,n}) + \varepsilon.
\end{equation}
Recall now that, for each $n \in \mathbb{N}$, $\mathbf{A}^n$ is a Nash equilibrium for the $n$-Lipschitz game and notice that the process constantly equal to zero is admissible. Hence, from (\ref{equazione_assurda!!!}), and using the coercivity condition (\ref{f>c}) and the integrability condition (\ref{H2:integrability}) in Assumption \ref{H2}, we find 
\begin{align*}
c\, \mathbb{E}^{\mathbb{P}}[|B_T^n|] \leq \mathcal{J}_\beta^{i}(B^n,A^{-i,n}) &< \mathcal{J}_\beta^{i}(\mathbf{A}^n)- \varepsilon \\
&\leq \mathcal{J}_\beta^i(0, A^{-i,n})  \leq \mathbb{E}^{\mathbb{P}_0}\left[ \int_0^T H^i((\pi_L)_t)\,dt + G^i((\pi_L)_T) \right] <\infty,
\end{align*}
which implies that
\begin{equation}\label{tightness_of_B}
\sup_{n} \mathbb{E}^{\mathbb{P}} [|B_T^n|] < \infty.
\end{equation}

With arguments analogous to those employed in the proof of \emph{Claim 1} of Theorem \ref{th}, from the tightness condition (\ref{tightness_of_B}) we deduce that there exists a subsequence of indexes (still denoted by $n$) and a probability measure $\tilde{\mathbb{P}}\in \mathcal{P}( \mathcal{C}_+^{Nd} \times \mathcal{D}^{k} \times \mathcal{D}_{\uparrow}^{(1+N) d} )$ such that the sequence $\mathbb{P} \circ (f,{L},{B}^n,\mathbf{A}^n)^{-1}$ converges weakly to $\tilde{\mathbb{P}}$.

Then, thanks again to an extension of Skorokhod's theorem (see Theorem 3 in \cite{D}), there exists a probability space $(\bar{\Omega},\bar{\mathcal{F}}, \bar{\mathbb{Q}})$, and, on it, a sequence
$
\{ (\bar{f}^n, \bar{L}^n,\bar{B}^n,\bar{\mathbf{A}}^n )\}_{n\in \mathbb{N}}
$
of $\mathcal{C}_+^{Nd} \times \mathcal{D}^{k} \times \mathcal{D}_{\uparrow}^{(1+N) d} $-valued random variables, and a $\mathcal{C}_+^{Nd} \times \mathcal{D}^{k} \times \mathcal{D}_{\uparrow}^{(1+N) d}$-valued random variable $(\bar{f},\bar{L},\bar{B},\bar{\mathbf{A}})$, such that  $\bar{\mathbb{Q}} \circ (\bar{f}^n, \bar{L}^n,\bar{B}^n,\bar{\mathbf{A}}^n)^{-1}=\tilde{\mathbb{P}}^n$ and  $\bar{\mathbb{Q}} \circ(\bar{f}, \bar{L},\bar{B},\bar{\mathbf{A}} )^{-1}=\tilde{\mathbb{P}}$.
 Furthermore, this representation is such that, for $\bar{\mathbb{Q}}$-almost all $\omega \in \bar{\Omega}$, we have
\begin{equation}\label{convunifoff_second}
\bar{f}^n(\omega)\rightarrow \bar{f}(\omega) \quad \text{uniformly on the interval } [0,T],
\end{equation}
as well as
\begin{equation}\label{convinmeasurelamda_seond}
( \bar{L}^n(\omega),\bar{B}^n(\omega),\bar{\mathbf{A}}^n(\omega) ) \rightarrow ( \bar{L}(\omega),\bar{B}(\omega),\bar{\mathbf{A}}(\omega) ) \quad \text{in the  measure $dt+\delta_T$ on $[0,T]$.}
\end{equation}

A rationale similar to that yelding (\ref{unifconvY}) can be employed to show that, $\bar{\mathbb{Q}}$-a.s.,
\begin{align}\label{limit_hg_epsilon}
\lim_n  & \int_0^T  h^i(\bar{L}_t^n,\bar{B}_t^n,  \bar{A}_t^{-i,n} )\, dt + g^i(\bar{L}_T^n, \bar{B}_T^n, \bar{A}_T^{-i,n} )  \\ \notag 
& =  \int_0^T h^i(\bar{L}_t, \bar{B}_t,, \bar{A}_t^{-i} )\, dt + g^i(\bar{L}_T, \bar{B}_T, \bar{A}_T^{-i} ) , \notag
\end{align}
where we have also used that  $h^i$ and $g^i$ are continuous.
Furthermore, thanks to the growth condition (\ref{growth_h_and_g_epsilon_Nash}), for $p>1$ as in Assumption \ref{H2}, we can find a suitable constant $\tilde{C}>0$ such that 
\begin{align}\label{uniform_integrability_epsilon}
\sup_n \, & \mathbb{E}^{\bar{\mathbb{Q}}} \bigg[ \bigg|  \int_0^T h^i(\bar{L}_t^n,\bar{B}_t^n,  \bar{A}_t^{-i,n} )\, dt + g^i(\bar{L}_T^n, \bar{B}_T^n, \bar{A}_T^{-i,n} ) \, \bigg|^p\, \bigg] \\ \notag
& \leq \tilde{C} \sup_n \,   \bigg( 1 + \mathbb{E}^{\mathbb{P}_0} \bigg[ \sup_{t\in [0,T]} |(\pi_L)_t|^{\gamma_1 p} \bigg] + \mathbb{E}^{\mathbb{P}} \left[ |\mathbf{A}_T^n|^{\gamma_2 p} \right] \bigg)	< \infty, 
\end{align}
where the integrability of the right-hand side follows from Condition (\ref{H2:integrability.of.the.sup}) and the estimate (\ref{stimafondamental}).
Finally, the limit in (\ref{limit_hg_epsilon}), together with the uniform integrability in (\ref{uniform_integrability_epsilon}), allows us to conclude that
\begin{align}\label{limit_integral_B}
\lim_n \mathbb{E}^{\bar{\mathbb{Q}}} & \left[ \int_0^T h^i(\bar{L}_t^n,\bar{B}_t^n,  \bar{A}_t^{-i,n} )\, dt + g^i(\bar{L}_T^n, \bar{B}_T^n, \bar{A}_T^{-i,n} ) \,\right] \\ \notag 
& = \mathbb{E}^{\bar{\mathbb{Q}}} \left[ \int_0^T h^i(\bar{L}_t, \bar{B}_t, \bar{A}_t^{-i} )\, dt + g^i(\bar{L}_T, \bar{B}_T, \bar{A}_T^{-i} ) \,\right]. \notag
\end{align}
With a similar reasoning we also find 
\begin{align}\label{limit_integral_A}
\lim_n \mathbb{E}^{\bar{\mathbb{Q}}} & \left[ \int_0^T h^i(\bar{L}_t^n, \bar{A}_t^{i,n}, \bar{A}_t^{-i,n} )\, dt + g^i(\bar{L}_T^n, \bar{A}_T^{i,n}, \bar{A}_T^{-i,n} ) \,\right] \\ \notag 
& = \mathbb{E}^{\bar{\mathbb{Q}}} \left[ \int_0^T h^i(\bar{L}_t, \bar{A}_t^i, \bar{A}_t^{-i} )\, dt + g^i(\bar{L}_T, \bar{A}_T^i, \bar{A}_T^{-i} ) \,\right]. \notag
\end{align}
Moreover,  Condition (\ref{H2:integrability.of.the.sup}) yields
\begin{equation}\label{stima.f^n.f}
    \sup_n \mathbb{E}^{\bar{\mathbb{Q}}} \bigg[  \sup_{t \in [0,T]} |\bar{f}_t^n|^{\alpha p} + \sup_{t \in [0,T]} |\bar{f}_t|^{\alpha p}    \bigg]= 2 \, \mathbb{E}^{\mathbb{P}_0}\bigg[ \sup_{t \in [0,T]} |{(\pi_f)}_t|^{\alpha p}  \bigg] < \infty.
\end{equation}
The latter, together with (\ref{stimafondamental}) and (\ref{stima.A}), allows to use Lemma \ref{existenceofasubsequenceforlimitofcosts} in Appendix \ref{App:MZ} in order to deduce that
\begin{equation*}
 \lim_n \mathbb{E}^{\bar{\mathbb{Q}}} \left[ \int_{0}^T \bar{f}_t^{i,n} \, d \bar{A}_t^{i,n} \, \right] =\mathbb{E}^{\bar{\mathbb{Q}}} \bigg[ \int_{[0,T]} \bar{f}_t^i \, d \bar{A}_t^i \, \bigg], 
\end{equation*}
which, together with (\ref{limit_integral_A}), gives
\begin{equation}\label{limitoffunctional_J}
\lim_n \mathcal{J}_\beta^{i}(\mathbf{A}^n) =\mathcal{J}_{\bar{\beta}}^i(\bar{A}^i,\bar{A}^{-i} ).
\end{equation}

Fix now $M \in \mathbb{N}$ and define the sequence of processes $\{ \bar{B}^{ n,M} \}_{n \in \mathbb{N}}$ by $\bar{B}_t^{n,M}:=\bar{B}_t^n \land M $ as well as the process $\bar{B}_t^M:=\bar{B}_t \land M$. Observe that, for each $n \in \mathbb{N}$, from (\ref{equazione_assurda!!!}) and the definition of $\bar{B}^{n,M}$ we have
\begin{equation}\label{equazione_assurda_M}
\mathcal{J}_{\bar{\beta}}^{i}(\bar{\mathbf{A}}^n) > \mathbb{E}^{\bar{\mathbb{Q}}}  \left[ \int_0^T h^i(\bar{L}_t^n, \bar{B}_t^n, \bar{A}_t^{-i,n} )\, dt + g^i(\bar{L}_T^n, \bar{B}_T^n, \bar{A}_T^{-i,n} ) \,\right] + \mathbb{E}^{\bar{\mathbb{Q}}}  \bigg[  \int_{[0,T]} \bar{f}_t^{i,n}\,d\bar{B}_t^{n,M} \bigg] + \varepsilon.
\end{equation}
Moreover, notice that the convergence established in (\ref{convinmeasurelamda_seond}) implies that,  $\bar{\mathbb{Q}}$-a.s., the sequence  $\{ \bar{B}^{n,M} \}_{n \in \mathbb{N}}$ converges to $\bar{B}^M$ in the measure $ dt + \delta_T $ on $[0,T]$.

Now, since the sequence $\{ \bar{B}^{ n,M} \}_{n \in \mathbb{N}}$ is bounded by the constant $M$, we can use again Lemma \ref{existenceofasubsequenceforlimitofcosts} in Appendix \ref{App:MZ} to deduce that
\begin{equation}\label{limit_M}
\lim_n \mathbb{E}^{\bar{\mathbb{Q}}}  \bigg[  \int_{[0,T]} \bar{f}_t^{i,n}\,d\bar{B}_t^{n,M} \bigg] = \mathbb{E}^{\bar{\mathbb{Q}}}  \bigg[  \int_{[0,T]} \bar{f}_t^{i}\,d\bar{B}_t^{M} \bigg]. 
\end{equation}
Hence, thanks to (\ref{limitoffunctional_J}), (\ref{limit_integral_B}) and (\ref{limit_M}), for each fixed $M$ we can pass to the limit in the inequality (\ref{equazione_assurda_M}), in order to obtain that
\begin{equation*}
\mathcal{J}_{\bar{\beta}}^{i}(\bar{\mathbf{A}}) \geq \mathbb{E}^{\bar{\mathbb{Q}}}  \left[ \int_0^T h^i(\bar{L}_t, \bar{B}_t, \bar{A}_t^{-i} )\, dt + g^i(\bar{L}_T, \bar{B}_T, \bar{A}_T^{-i} ) \,\right] + \mathbb{E}^{\bar{\mathbb{Q}}}  \bigg[  \int_{[0,T]} \bar{f}_t^{i}\,d\bar{B}_t^{M} \bigg] + \varepsilon.
\end{equation*}
Finally,  by the monotone convergence theorem, we can take the limit as $M\to \infty$ in the latter inequality to deduce that  
\begin{equation}\label{equazione_assurda_limite}
\mathcal{J}_{\bar{\beta}}^i(\bar{A}^i,\bar{A}^{-i} ) \geq \mathcal{J}_{\bar{\beta}}^i( \bar{B} , \bar{A}^{-i} ) + \varepsilon.
\end{equation}

On the other hand,  the probability measure $\bar{\mathbb{Q}} \circ(\bar{f}, \bar{L},\bar{\mathbf{A}} )^{-1}$ is an accumulation point of the sequence ${\mathbb{P}} \circ({f}, {L},{\mathbf{A}}^n )^{-1}$, and hence, by Theorem \ref{th}, the couple $(\bar{\beta}, \bar{\mathbf{A}})$ is a weak Nash equilibrium of the monotone-follower game, with $\bar{\beta}:=(\bar{\Omega},\bar{\mathcal{F}}, \bar{\mathbb{Q}},\bar{f},\bar{L})$. Moreover,  $\bar{B}$ is an admissible strategy associated to the basis $\bar{\beta}$; this implies that 
$$
\mathcal{J}_{\bar{\beta}}^i(\bar{A}^i,\bar{A}^{-i} ) \leq \mathcal{J}_{\bar{\beta}}^i( \bar{B} , \bar{A}^{-i} ),
$$
which, together with (\ref{equazione_assurda_limite}), leads  to a contradiction, and thus completes the proof. 
\end{proof}

\begin{remark}
Theorem \ref{Existence_of_epsilon} can also be understood in a different way. Fix a weak Nash equilibrium $(\bar{\beta},\bar{\mathbf{A}})$ which is an accumulation point of a sequence of Nash equilibria of the $n$-Lipschitz game on a fixed basis $\beta$, and define $$\mathcal{V}=(\mathcal{V}^1,...,\mathcal{V}^N):=(\mathcal{J}_{\bar{\beta}}^1(\bar{\mathbf{A}}),...,\mathcal{J}_{\bar{\beta}}^N(\bar{\mathbf{A}})).$$Then, $\mathcal{V}$ is a \emph{Nash equilibrium payoff} of the monotone-follower game (see, e.g., Definition 2.7 in \cite{Cardaliaguet2004}, or \cite{Lin2012}), in the sense that, for each $\varepsilon>0$, there exists $\mathbf{A}^{\varepsilon} \in \mathcal{A}_\beta^N$ such that, for each $i=1,...,N$, we have:
\begin{enumerate}
\item   $\mathcal{J}_\beta^i(A^{i,\varepsilon},A^{-i,\varepsilon}) \leq \mathcal{J}_\beta^i(B^i,A^{-i,\varepsilon}) + \varepsilon \quad$  for each$\quad B^i \in \mathcal{A}_\beta$;
\item $|\mathcal{J}_\beta^i(\mathbf{A}^\varepsilon) - \mathcal{V}^i| \leq \varepsilon$.
\end{enumerate}
Moreover, Theorem \ref{Existence_of_epsilon} shows that the Nash equilibrium payoff $\mathcal{V}$ is such that, for each $\varepsilon>0$, the profile strategy $\mathbf{A}^\varepsilon$, which satisfies the conditions of the definition above, can be chosen as a Nash equilibrium of the $n$-Lipschitz game, for $n$ large enough.
\end{remark}
\begin{remark}
Notice that the submodularity conditions (\ref{H1:2}) and (\ref{H2:3}) in Assumption \ref{H1} are not necessarily needed in the proof of Theorem \ref{th} and \ref{Existence_of_epsilon}. Indeed, only the requirement that, for each $n \in \mathbb{N}$, there exists a Nash equilibrium for the $n$-Lipschitz game is needed. 
The latter games can be seen as stochastic differential games, where the set of strategies is the set of progressively measurable stochastic processes $u^i:\Omega \times [0,T] \rightarrow [0,n]^d$, with degenerate dynamics $A_t^i= \int_0^t u_s^i ds$. This fact suggests that, whenever the submodularity requirement does not hold, one might exploit, on a case by case basis, existence results on equilibria for sochastic differential games (see, e.g., \cite{Carmona2016} and references therein for results on stochastic differential games).
\end{remark}
\end{section}

%%%%%%%%%%%%%%%%%%%%%%%%%%%%%%%%%%%%%%%%%%%%%%%%%%%%%%%%%%%%%%%
%%%        APPLICATION AD EXAMPLES
%%%%%%%%%%%%%%%%%%%%%%%%%%%%%%%%%%%%%%%%%%%%%%%%%%%%%%%%%%%

\begin{section}{Applications and Examples}
\label{Section.Application.Examples}

%% STOCHASTIC DIFFERENTIAL GAMES

\subsection{Existence of Equilibria in a Class of Stochastic Differential Games.} 

This subsection is devoted to show that Theorem \ref{ExistenceNashEquilibrium} applies to deduce existence of open loop Nash equilibria in stochastic differential games with  singular controls, whenever a certain structure is preserved by the dynamics. For the sake of illustration, we propose the following model.

Fix a filtered probability space $(\Omega, \mathcal{F},\mathbb{F}, \mathbb{P})$ satisfying the usual conditions and consider on it $N$ standard $\mathbb{F}$-Brownian motions $W^i$. Suppose to be given, for $i=1,...,N$, measurable functions $g^i,h^i:\R^k \times \R^{N} \rightarrow \R$, as well as constants $\mu^i,\sigma^i \in \R$ and continuous $\mathbb{F}$-adapted stochastic processes $f^i: \Omega \times [0,T]  \rightarrow [0,\infty)$. Assume moreover to be given an $\mathbb{F}$-adapted process $L:[0,T]\times\Omega \rightarrow \mathbb{R}^k$ with c\`adl\`ag components. The set of admissible strategies $\mathcal{A}$ is defined  as the set of nondecreasing, nonnegative, c\`adl\`ag, $\mathbb{F}$-adapted, $\R$-valued stochastic processes, whereas $\mathcal{A}^N:=\bigotimes_{i=1}^N \mathcal{A}$ denotes the set of asmissible profile strategies.

We consider the $N$-player stochastic differential game of  singular controls in which, for $i=1,...,N$, player $i$  chooses an admissible strategy $\xi^i \in \mathcal{A}$ to control her private state, which evolves according to the stochastic differential equation 
\begin{equation}\label{SDE}
dX_t^i=\mu^i X_t^i\, dt + \sigma^i X_t^i \,dW_t^i+d\xi_t^i, \quad t\in [0,T], \quad X_{0-}^i=x_0^i>0,
\end{equation}
in order to minimize her expected cost
\begin{equation*}
\mathcal{J}^i(\xi^i,\xi^{-i}):= \mathbb{E} \bigg[ \int_0^T h^i(L_t,X_t^i,X_t^{-i}) dt + g^i(L_T,X_T^i,X_T^{-i})+\int_{[0,T]} f_t^i d\xi_t^i  \bigg].
\end{equation*}
Observe that, for $i=1,...,N$, the solution to equation (\ref{SDE}) is given by
\begin{equation}\label{rapr_of_X}
X_t^i=E_t^i \bigg[ x_0^i+ \int_{[0,t]} \frac{1}{E_s^i} d\xi_s^i \bigg]=E_t^i \left[ x_0^i+ \bar{\xi}_t^i \right]  ,
\end{equation}
where the processes $\{ E_t^i \}_{t \in [0,T]}$ and $\{ \bar{\xi}_t^i \}_{t \in [0,T]}$ are defined by
\begin{equation}\label{def_xi}
 E_t^i:=\exp \left[ \left( \mu^i-\frac{(\sigma^i)^2}{2} \right) t+\sigma^i W_t^i  \right] \quad \text{and} \quad \bar{\xi}_t^i:= \int_{[0,t]} \frac{1}{E_s^i} d\xi_s^i.
\end{equation}

\begin{assumption}\label{H_diff_game}
Let $h^i$ and $g^i$ satisfy Assumption \ref{H1}. Suppose moreover that: 
\begin{enumerate}
\item  for each $i=1,...,N$, there exist functions $\widetilde{H}^i,\widetilde{G}^i:\R^k \times \R \rightarrow [0,\infty)$ such that 
$$
h^i(l,x^i,x^{-i}) \leq  \widetilde{H}^i(l,x^i) \quad \text{and} \quad g^i(l,x^i,x^{-i}) \leq  \widetilde{G}^i(l,x^i), \quad \text{for each} \quad (l,x) \in \R^k \times \R^N,
$$
with 
$$
\mathbb{E} \left[ \int_0^T \widetilde{H}^i(L_t,x_0 E_t^i)\, dt  +  \widetilde{G}^i(L_T,x_0 E_T^i )   \right] < \infty;
$$
\item there exists a constant $k_1$ such that, for each $i=1,...,N$, we have $g^i(l,x) \geq k_1 x^i$ for each $(l,x) \in \R^k \times \R^N$.
\end{enumerate}
\end{assumption}

\begin{theorem}
\label{StochasticDifferentialGame}
Under Assumption \ref{H_diff_game}, there exists an open-loop Nash equilibrium of the previously introduced stochastic differential game.
\end{theorem}
\begin{proof}
Thanks to (\ref{rapr_of_X}), the cost functional of player $i$ can be rewritten in terms of  $\bar{\xi}^i$ (cf. (\ref{def_xi})), that is 
\begin{align}\label{expression_functional}
 \mathcal{J}^i(\xi^i,\xi^{-i}) &= \mathbb{E} \bigg[ \int_0^T h^i \left(L_t, E_t^i \left[ x_0^i+ \bar{\xi}_t^i \right],   \left\{ E_t^j \left[ x_0^j+ \bar{\xi}_t^j \right] \right\}_{j\ne i} \right) dt \\ \notag
& \quad \quad + g^i \left(L_T, E_T^i \left[ x_0^i+ \bar{\xi}_T^i \right] ,\left\{ E_T^j \left[ x_0^j+ \bar{\xi}_T^j \right] \right\}_{j\ne i} \right) + \int_{[0,T]} f_t^i E_t^i \,d\bar{\xi}_t^i  \bigg]. \notag
\end{align} 
This leads to define the new functions $\bar{h}^i,\bar{g}^i:\R^k \times (0,\infty)^N\times \R^N \rightarrow [0,\infty)$ by
$$
\begin{matrix}
& \bar{h}^i(l,e,z^i,z^{-i}):=h^i(l,e^i[x_0^i+z^i], \{ e^j[x_0^j+z^j] \}_{j \ne i} ) \\
& \bar{g}^i(l,e,z^i,z^{-i}):=g^i(l,e^i[x_0^i+z^i], \{ e^j[x_0^j+z^j] \}_{j \ne i} ), \\
\end{matrix}
$$
as well as the continuous processes $\bar{f}^i: \Omega \times [0,T] \rightarrow \R$ by $\bar{f}_t^i:=f_t^i \, E_t^i$. These definitions allows us to introduce new cost functionals in terms of new profile strategies  $\zeta=(\zeta^1,...,\zeta^N) \in \mathcal{A}^N$ setting
$$
\bar{\mathcal{J}}^i({\zeta}^i,{\zeta}^{-i}):= \mathbb{E} \bigg[ \int_0^T \bar{h}^i(L_t,E_t,\zeta_t^i,\zeta_t^{-i}) dt + \bar{g}^i(L_T,E_T,\zeta_T^i,\zeta_T^{-i})+\int_{[0,T]} \bar{f}_t^i  d\zeta_t^i  \bigg].
$$
Notice that, by (\ref{expression_functional}) and the definition of $\bar{\xi}^i$ in (\ref{def_xi})  as a function of $\xi^i$, we have that
$$
\bar{\mathcal{J}}^i(\bar{\xi}^i,\bar{\xi}^{-i})=\mathcal{J}^i(\xi^i,\xi^{-i}), \quad \forall \,\xi \in \mathcal{A}^N, \quad \forall \, i \in \{1,...,N\}.
$$
Furthermore, for each $\zeta \in \mathcal{A}^N$ there exists a unique $\xi \in \mathcal{A}^N$ such that $\zeta^i=\bar{\xi}^i$ for each $i\in\{1,...,N\}$. This means that solving the  stochastic differential game in the class of profile strategies $\xi \in \mathcal{A}$ and with cost  functionals $\mathcal{J}^i$ is equivalent to solve the monotone-follower game for $\zeta \in \mathcal{A}$ and cost functionals  $\bar{\mathcal{J}}^i$. 
The rest of the proof is then mainly devoted to show that the costs $\bar{h}^i$ and  $\bar{g}^i$, together with the processes $\bar{f}^i$, satisfy the conditions of Theorem \ref{ExistenceNashEquilibrium}.

Since the functions $h^i$ and $g^i$ satisfy Assumption \ref{H1}, for each $(l,e,z^{-i}) \in \R^k \times (0,\infty)^N \times \R^{N-1}$ the functions $\bar{h}^i(l,e,\cdot,z^{-i})$ and $\bar{g}^i(l,e,\cdot,z^{-i})$ are clearly continuous and strictly convex. Moreover, for $(l,e) \in \R^k \times (0,\infty)^N$ and $z,\bar{z} \in \R^N$ such that $z \leq \bar{z}$, we have $e^j[x_0^j+z^j] \leq e^j[x_0^j+\bar{z}^j]$ for each $j=1,...,N$, since the components of $e$ are positive. Therefore, because  $h^i$ has decreasing differences, we deduce that
\begin{align*}
\bar{h}^i(l,e,\bar{z}^i,z^{-i})&-\bar{h}^i(l,e,z^i,z^{-i}) \\
&= h^i(l,e^i[x_0^i+\bar{z}^i], \{ e^j[x_0^j+z^j] \}_{j \ne i} )-h^il,(e^i[x_0^i+z^i], \{ e^j[x_0^j+z^j] \}_{j \ne i} ) \\
& \geq h^i(l,e^i[x_0^i+\bar{z}^i], \{ e^j[x_0^j+\bar{z}^j] \}_{j \ne i} )-h^i(l,e^i[x_0^i+z^i], \{ e^j[x_0^j+\bar{z}^j] \}_{j \ne i} ) \\
&=\bar{h}^i(l,e,\bar{z}^i,\bar{z}^{-i})-\bar{h}^i(l,e,z^i,\bar{z}^{-i}),
\end{align*}
which means that $\bar{h}^i$ has decreasing difference as well. In the same way it is possible to show that $\bar{g}^i$ has decreasing differences, and this allows to conclude that the functions $\bar{h}^i$ and $\bar{g}^i$ satisfy Assumption \ref{H1}. Moreover, thanks to (1) in  Assumption \ref{H_diff_game}, Condition \ref{boundedness_of_h,g} is clearly satisfied  with $r^i(\zeta)=0$ for each $\zeta \in \mathcal{A}^N$. 

We prove now that the functionals $\bar{\mathcal{J}}^i$ satisfy a slightly different version of Condition \ref{coercivity_condition}. The superlinear condition (2) in Assumption \ref{H_diff_game} implies that
\begin{align*}
\bar{J}^i(\zeta^i, \zeta^{-i}) \geq & \mathbb{E}\left[  \bar{g}^i(L_T,\zeta_T^i, \zeta_T^{-i})  \right] =  \mathbb{E}\left[  g^i \left(L_T, E_T^i \left[ x_0^i+ {\zeta}_T^i \right] ,\left\{ E_T^j \left[ x_0^j+ {\zeta}_T^j \right] \right\}_{j\ne i} \right)  \right] \\
& \geq k_1  \mathbb{E}\left[  E_T^i \left[ x_0^i+ {\zeta}_T^i \right] \right] \geq  k_1  \mathbb{E}\left[  E_T^i \, {\zeta}_T^i \right] =  k_1 \, \mathbb{E}[  E_T^i ]  \,   \mathbb{E}^{\tilde{\mathbb{P}}^i}\left[  {\zeta}_T^i \right],
\end{align*}
where $\tilde{\mathbb{P}}^i$ is the probability measure on $(\Omega, \mathcal{F})$ given by  $$ d \tilde{\mathbb{P}}^i:= \frac{ E_T^i}{\mathbb{E}[  E_T^i ]}\, d \mathbb{P},$$ and equivalent to $\mathbb{P}$.

We can therefore apply Theorem \ref{ExistenceNashEquilibrium} (in fact a slightly different version of it, in which the expectation in Condition \ref{coercivity_condition} is replaced by the expectation under an equivalent probability measure)  to deduce existence of a Nash equilibrium $\hat{\zeta}=(\hat{\zeta}^1,..., \hat{\zeta}^N)$ of the monotone-follower game with cost functionals $\bar{\mathcal{J}}^i$. Hence the process $\hat{\xi}=(\hat{\xi}^1,..., \hat{\xi}^N)$ defined by 
$$
\hat{\xi}_t^i:= \int_{[0,t]} E_s^i\, d\hat{\zeta}_s^i
$$
is an open-loop Nash equilibrium of the stochastic differential game.
\end{proof}

\begin{remark}
The same arguments employed in the proof of Theorem \ref{StochasticDifferentialGame} apply if we replace the dynamics of the controlled geometric Brownian motion in (\ref{SDE}) by the dynamics of a controlled  Ornstein–Uhlenbeck process
$$
dX_t^i=\theta^i(\mu^i- X_t^i) \, dt + \sigma^i \,dW_t^i+d\xi_t^i, \quad t\in [0,T], \quad X_{0-}^i=x_0^i>0,
$$
for some parameters $\theta^i,\sigma^i>0$ and $\mu^i \in \R$. Mean-reverting dynamics (as the  Ornstein–Uhlenbeck one) find important application in the energy and commodity markets (see, e.g., \cite{Benth&Kholodnyi&Laurence14} or Chapter 2 in \cite{Lutz10}).
\end{remark}

%%%%%%%%%%%%%%%%%%%%%%%%%%%%%%%%%%%%%%%%%%%%%%%%%%
%%                 ALGORITHM
%%%%%%%%%%%%%%%%%%%%%%%%%%%%%%%%%%%%%%%%%%%%%%%%%%%%%%%%%%

\subsection{An Algorithm to Approximate the Least Nash Equilibrium} 

In this subsection we prove that, also in our setting, the algorithm introduced by Topkis (see Algorithm II in \cite{To}) for submodular games converges to the least Nash equilibrium of the game.  

According to the notation of Section 2, define the sequence of processes $\{ \mathbf{R}^n \}_{n \in \mathbb{N}} \subset \mathcal{A}^N$ in the following way: 
\begin{itemize}
    \item  $\mathbf{R}^0 = 0 \in \mathcal{A}^N$;
    \item  for each $n \geq 1$, set $\mathbf{R}^{n+1}:=\mathbf{R}(\mathbf{R}^n)$.
 \end{itemize}

\begin{theorem} 
Suppose that the assumptions of Theorem \ref{ExistenceNashEquilibrium} hold. Assume, moreover, that there exists a constant $C>0$ such that, for each $i=1,...,N$,
\begin{equation}\label{condition_algorithm}
h^i(l,a)+g^i(l,a)\leq C(1+|a|), \quad \forall \, (l,a) \in \R^k \times \R^{Nd} \quad \text{and} \quad |f_t^i| \leq C, \quad \forall \, t \in [0,T], \quad \mathbb{P}-a.s.
\end{equation}
Then the sequence $\{ \mathbf{R}^n \}_{ n \in \mathbb{N} }$ is monotone increasing in the lattice $(\mathcal{A}^N, \preccurlyeq^N )$ and it converges to the least Nash equilibrium of the game.
\end{theorem}
\begin{proof}
Since the map $\mathbf{R}:\mathcal{A}^N\rightarrow \mathcal{A}^N$ is increasing (cf.\ \emph{Step 2}  in the proof of Theorem \ref{ExistenceNashEquilibrium}), the sequence $\{ \mathbf{R}^n \}_{ n \in \mathbb{N} }$ is clearly monotone increasing with respect to the order relation in $\mathcal{A}^N$. 

Define now the process $\mathbf{S}:=(S^1,...,S^N) \in \mathcal{A}_\infty^N$ as the least upper bound of the sequence $\{ \mathbf{R}^n \}_{ n \in \mathbb{N} }$ in the lattice $(\mathcal{A}_\infty^N, \preccurlyeq^N )$. 
Recall the construction of $\mathbf{S}$ and $\tilde{\mathbf{S}}$ (cf.\ (\ref{sup_count}) and (\ref{least_upper_bound}) in \emph{Step 3} in the proof of Theorem \ref{ExistenceNashEquilibrium}). Notice that, since the sequence $\{ \mathbf{R}^n \}_{ n \in \mathbb{N} }$ is increasing in the lattice  $(\mathcal{A}^N,\preccurlyeq^N)$,  there exists a $\mathbb{P}$-null set $\mathcal{N}$ such that
$$ 
\tilde{\mathbf{S}}_q(\omega)=\lim_{n  } \mathbf{R}_q^n (\omega)=\sup_{n } \mathbf{R}_q^n (\omega), \quad \forall \, q \in Q:=([0,T] \cap \mathbb{Q}) \cup \{ T \}, \quad \forall \, \omega \in \Omega \setminus \mathcal{N}.
$$
Take now $\bar{t} \in (0,T)$ and $\omega \in \Omega \setminus \mathcal{N}$. If $\mathbf{R}_{\bar{t}}^n(\omega)$ does not converge to $\mathbf{S}_{\bar{t}}(\omega)$, then we find $\varepsilon>0$ such that,
$$
\tilde{\mathbf{S}}_q(\omega) + \varepsilon = \sup_{n } \mathbf{R}_q^n(\omega) + \varepsilon \leq \sup_{n} \mathbf{R}_{\bar{t}}^n(\omega) + \varepsilon \leq \mathbf{S}_{\bar{t}}(\omega). 
$$
for each $ q \in Q$ such that $q <\bar{t}$.
This implies that
$
\mathbf{S}_{{\bar{t}}-}(\omega) +\varepsilon \leq \mathbf{S}_{\bar{t}}(\omega)$, which means that ${\bar{t}}$ is in the set $\mathcal{I}(\omega)$ of discontinuity points of $\mathbf{S}(\omega)$.
Thus, we conclude that there exists a $\mathbb{P}$-null set $\mathcal{N}$ such that,
\begin{equation}\label{limit_of_R}
\mathbf{S}_t(\omega)=\lim_{n} \mathbf{R}_t^n(\omega) \quad \forall \, t \in [0,T]\setminus \mathcal{I}(\omega), \quad \forall \, \omega \in \Omega \setminus  \mathcal{N},
\end{equation}
since, for each $\omega \in \Omega \setminus \mathcal{N}$, the latter convergence is verified in $T$ by the definition of $\mathbf{S}_T$. 

We next show that the limit point $\mathbf{S}$ is a Nash equilibrium. By \emph{Step 1} in the proof of Theorem \ref{ExistenceNashEquilibrium}, we know that there exists a suitable constant $\widetilde{C}$ such that, for each $n \in \mathbb{N}$,  $\mathbb{E} [|\mathbf{R}_T^n|] \leq \widetilde{C}$. Hence, by the monotone convergence theorem, we deduce that
\begin{equation}\label{bound_for_S_T}
\mathbb{E}[|\mathbf{S}_T|] \leq \widetilde{C},
\end{equation}
which in turn implies that $\mathbf{S} \in \mathcal{A}^N.$
Fix then $i \in \{1,...,N\}$ and  $B^i \in \mathcal{A}$. If $\mathbb{E}[|B_T^i|]= \infty$, then, by the coercivity condition (\ref{coercivity_condition}), we would  automatically have
$
\mathcal{J}^i(S^i, S^{-i}) \leq \mathcal{J}^i(B^i, S^{-i})= \infty. 
$
Hence, without loss of generality, we can assume that 
\begin{equation}\label{B-bounded}
\mathbb{E}[|B_T^i|] < \infty.
\end{equation}
Now, since $R^{i,n+1}$ minimizes $\mathcal{J}^i(\cdot, R^{-i,n})$,  for each $n \in \mathbb{N}$ we can write
\begin{align*}
\mathbb{E} & \bigg[ \int_0^T h^i(L_t, R_t^{i,n+1},R_t^{-i,n})\, dt  + g^i(L_T,R_T^{i,n+1},R_T^{-i,n}) + \int_{[0,T]} f_t^i \, dR_t^{i,n+1} \bigg] \\
&\leq \mathbb{E} \bigg[ \int_0^T h^i(L_t, B_t^i,R_t^{-i,n})\, dt  + g^i(L_T,B_T^i,R_T^{-i,n}) + \int_{[0,T]} f_t^i \, dB_t^i \bigg].
\end{align*}
Moreover, the limit in (\ref{limit_of_R}), together with conditions (\ref{condition_algorithm}) and the estimates (\ref{bound_for_S_T}) and (\ref{B-bounded}), allows us to invoke the dominated convergence theorem and to take the limit as $n$ goes to infinity in the last inequality in order to deduce that $\mathcal{J}^i(S^i, S^{-i})\leq\mathcal{J}^i(B^i, S^{-i})$. Hence $\mathbf{S}$ is a Nash equilibrium.

Finally, we prove that $\mathbf{S}$ is the least Nash equilibrium. Suppose that $\bar{\mathbf{S}}$ is another Nash equilibrium. By definition we have $\mathbf{R}^0 = 0 \preccurlyeq^N \bar{\mathbf{S}}$. If, for an arbitrary $n \in \mathbb{N}$, we have $\mathbf{R}^n \preccurlyeq^N \bar{\mathbf{S}}$, then, since the map $\mathbf{R}$ is increasing and $\bar{\mathbf{S}}$ is a fixed point of $\mathbf{R}$, we have $\mathbf{R}^{n+1} = \mathbf{R}(\mathbf{R}^n) \preccurlyeq^N \mathbf{R}(\bar{\mathbf{S}})=\bar{\mathbf{S}}$. Hence, by induction, we deduce that $\mathbf{R}^n \preccurlyeq^N \bar{\mathbf{S}}$ for each $n \in \mathbb{N}$, which in turn implies that $\mathbf{S} \preccurlyeq^N \bar{\mathbf{S}}$, since $\mathbf{S}$ is the least upper bound of the sequence $\{ \mathbf{R}^n \}_{n \in \mathbb{N} }$.
\end{proof}
\end{section}

%%%%%%%%%%%%%%%%%%%%%%%%%%%%%%%%%%%%%%%%%%%%%%%%%%%%%%%%%%%%%%%%%%%%%%%%%%%%%%%%%%%%%%%%
%%%%%%%%%%%%%%%%%%%%%%%%%%%%%%%%%%%%%%%%%%%%%%%%%%%%%%%%%%%%%%
\appendix

\section{Meyer-Zheng Convergence}
\label{App:MZ}
In this appendix we recall some fact about the so-called Meyer-Zheng topology (see \cite{MZ}) and we provide some results concerning the tightness of c\`adl\`ag processes in such a topology.
\smallbreak
\emph{Pseudopath topology.} Recall that we have defined (cf.\ Subsection \ref{subsection.Existence.Approximation.Weak.Nash.Equilibria}) the pseudopath topology $\tau_{pp}^{\text{\emph{\tiny T}}}$ on the space $\mathcal{D}^m$ as the topology induced by the convergence in the measure $dt+\delta_T$ on the interval $[0,T]$, where $dt$ denotes the Lebesgue measure and $\delta_T$ denotes the Dirac measure at the terminal point $T$. Notice that we introduce the pseudo-path topology through its characterization proved in Lemma 1 in \cite{MZ}.
 Observe that the topology  $\tau_{pp}^{\text{\emph{\tiny T}}}$ is metrizable. If $\{ x^n \}_{n \in \mathbb{N}}$ is a sequence of functions in $ \mathcal{D}^m$ converging to a function $x\in \mathcal{D}^m$ in the pseudopath topology $\tau_{pp}^{\text{\emph{\tiny T}}}$, then we have that (see, e.g., Appendix A.3. at p. 116 in \cite{LZ})
\begin{equation}\label{charatconv}
\lim_n \int_0^T \phi(s,x_s^n) \, ds = \int_0^T \phi(s,x_s) \, ds, \quad \text{and} \quad \lim_n x_T^n=x_T,
\end{equation}
for each bounded continuous function $\phi:[0,T] \times	\mathbb{R}^{m} \rightarrow \mathbb{R}$.
\smallbreak
\emph{Meyer-Zheng topology and tightness criteria.}
The \emph{Meyer-Zheng topology} on $\mathcal{P}(\mathcal{D}^m)$ is the topology of weak convergence of probability measures on the topological space $(\mathcal{D}^m, \tau_{pp}^{\text{\emph{\tiny T}}})$.

For a given filtered probability space $(\Omega,\mathcal{F},\mathbb{F},\mathbb{P})$ consider a c\`adl\`ag process $X: \Omega \times [0,T] \rightarrow \mathbb{R}^m$, and consider the \emph{conditional variation} of $X$ over the interval $[0,T]$, defined as
\begin{equation}\label{conditionalvariation_T}
V_T^{\mathbb{P}}(X):= \sup \sum_{i=1}^n \mathbb{E}\left[ \left|  \mathbb{E}[X_{t_i}-X_{t_{i-1}}|\mathcal{F}_{t_{i-1}} ] \right| \right] + \mathbb{E}[|X_{t_n}|],
\end{equation}
where the supremum is taken over all the partitions $0=t_0<...<t_n \leq T$, $n\in \mathbb{N}$.

We finally prove, for the sake of completeness, a slightly different version of the classical Meyer-Zheng tightness criterion (see Theorem 4 at p.\ 360 in \cite{MZ}), that is useful in many occasions during our study. 
Notice that, differently to Theorem 34 at p.\ 116 in \cite{LZ}, the next lemma allows us to handle a stochastic cost of control $f$.
\begin{lemma}\label{MeyerZheng} The following tightness criteria hold true.
\begin{enumerate}
\item Let $\{ X^n \}_{n\in \mathbb{N}}$ be a sequence of $\mathbb{R}^m$-valued c\`adl\`ag processes defined  on $[0,T]$ such that
$$
\sup_n V_T^{\mathbb{P}}(X^n)<\infty.
$$
Then $\{\mathbb{P} \circ X^n \}_{n\in \mathbb{N}}$ is tight in $ \mathcal{P}(\mathcal{D}^m)$.
\item Let $\{ X^n \}_{n\in \mathbb{N}}$ be a sequence of nondecreasing, nonnegative, $\mathbb{R}^m$-valued c\`adl\`ag processes defined on $[0,T]$ such that
$$
 \sup_n \mathbb{E} [|X_T^n|]<\infty.
$$
Then $\{\mathbb{P} \circ X^n \}_{n\in \mathbb{N}}$ is tight in $\mathcal{P}( \mathcal{D}_{\uparrow}^m)$.
\end{enumerate}
\end{lemma}
\begin{proof} We will prove only the claim (1), since the proof of claim (2) follows by an analogous rationale. 

Let $\mathcal{D}^m[0,\infty)$ be the space of $\mathbb{R}^m$-valued c\`adl\`ag functions on $[0,\infty)$, with the Borel $\sigma$-algebra generated by the Skorokhod topology. On the half line $[0,\infty)$, consider the measure $\lambda $ given by $d\lambda:=e^{ -t} dt$.
 On $\mathcal{D}^m[0,\infty)$ consider the \emph{pseudopath topology} $\tau_{pp}$; that is, the topology  induced by the convergence in the measure $\lambda$ on the interval $[0,\infty)$. 
Define, moreover, the space $\widetilde{\mathcal{D}}^m[0,\infty)$ as the set of elements of $\mathcal{\mathcal{D}}^m[0,\infty)$ which are constant on $[T,\infty)$, and notice that $\widetilde{\mathcal{D}}^m[0,\infty)$ is a closed subset of $\mathcal{\mathcal{D}}^m[0,\infty)$. Also, observe that the \emph{extension map} $\Psi:\mathcal{D}^m \rightarrow  \widetilde{\mathcal{D}}^m[0,\infty)$, defined by \begin{equation}\label{extension.map} \Psi(x)_t:=\begin{cases}
x_t & \text{if} \quad t \in [0,T] \\
x_T & \text{if} \quad t \in (T,\infty),
\end{cases} \end{equation}
is an omeomorphism between the topological spaces $(\mathcal{D}^m,\tau_{pp}^{\text{\emph{\tiny T}}})$ and $(  \widetilde{\mathcal{D}}^m[0,\infty), \tau_{pp})$. 

Now, using the uniform boundedness of $V_T^{\mathbb{P}}(X^n)$, we notice that the sequence $\Psi(X^n)$ satisfies the requirement of Theorem 4 in \cite{MZ},  and, as shown in its proof, it follows that the sequence $\{\mathbb{P} \circ \Psi(X^n) \}_{n\in \mathbb{N}}$ is tight in $ \mathcal{P}(\mathcal{D}^m[0,\infty))$.
Furthermore, since  $\widetilde{\mathcal{D}}^m[0,\infty)$ is a closed subset of $\mathcal{\mathcal{D}}^m[0,\infty)$, we have that $\{\mathbb{P} \circ \Psi(X^n) \}_{n\in \mathbb{N}}$ is tight in $\mathcal{P}(\widetilde{\mathcal{D}}^m[0,\infty))$. 
Finally, since the map $\Psi$ is an omeomorphism, we conclude that the sequence $\{\mathbb{P} \circ X^n \}_{n\in \mathbb{N}}$ is tight in $\mathcal{P}( \mathcal{D}^m)$ in the Meyer-Zheng topology.
\end{proof}

We finally summarize in a lemma a result on the convergence of stochastic integrals.

\begin{lemma}
\label{existenceofasubsequenceforlimitofcosts}
Let $\{ F^n \}_{n \in \mathbb{N}}$ be a sequence of $\mathbb{R}^m$-valued continuous processes which converges $\mathbb{P}$-a.s.\ to an $\mathbb{R}^m$-valued continuous process $F$  uniformly on $[0,T]$. Let $\{ X^n \}_{n \in \mathbb{N}}$ be a sequence of nondecreasing, nonnegative, $\mathbb{R}^m$-valued c\`adl\`ag processes defined on $[0,T]$, which converges $\mathbb{P}$-a.s. to nondecreasing, nonnegative, $\mathbb{R}^m$-valued cadlag process $X$ in the pseudopath topology $\tau_{pp}^{\text{\tiny T}}$. Suppose, moreover, that there exists two constant $\alpha, p  >1 $ such that
\begin{equation}\label{stima.su.tutto}
\sup_n \mathbb{E}\left[ \sup_{t \in [0,T]} \left( |F_t^n|^{\alpha p} + |F_t|^{\alpha p} \right) + |X_T^n|^{ \frac{ \alpha p}{p-1}} + |X_T|^{ \frac{ \alpha p}{p-1}}    \right] < \infty.
\end{equation}
Then 
\begin{equation}\label{limitof_f_dA}
 \lim_n \mathbb{E} \bigg[ \int_{[0,T]} {F}_t^n \, d X_t^n \, \bigg] = \mathbb{E} \bigg[ \int_{[0,T]} {F}_t \, d X_t \, \bigg]. 
 \end{equation}
\end{lemma}

\begin{proof} 
We will prove that for each subsequence of indexes there exists a further subsequence for which the limit in (\ref{limitof_f_dA}) holds true.

Consider then a subsequence of indexes (not relabeled). From Condition (\ref{stima.su.tutto}), Hölder's inequality with $p$ as in the assumptions easily reveals that
\begin{equation}\label{uniformintegraaplpha}
\sup_n \mathbb{E} \bigg[ \bigg| \int_{[0,T]} {F}_t^{n} d {X}_t^{n} \bigg|^\alpha \, \bigg]+\sup_n \mathbb{E} \bigg[ \bigg| \int_{[0,T]} {F}_t d {X}_t^{n} \bigg|^\alpha \, \bigg] < \infty.
\end{equation}
Since $\alpha>1$, by the reflexivity of $\mathbb{L}^\alpha(\mathbb{P})$, there exists a subsequence of indexes $n_j$ and a random variable $Z \in \mathbb{L}^\alpha(\mathbb{P})$, for which
\begin{equation}\label{limit_Z}
\lim_j \mathbb{E} \bigg[ \int_{[0,T]} {F}_t^{n_j} d {X}_t^{n_j} \bigg]=\lim_j \mathbb{E} \bigg[ \int_{[0,T]} {F}_t \,  d {X}_t^{n_j} \bigg]= \mathbb{E}[Z],
\end{equation}
where the equality of the two limits follows from the $\mathbb{P}$-a.s.\ uniform convergence of $F^n$ to $F$ and from the integrability condition (\ref{stima.su.tutto}). 

Next, since by Condition (\ref{stima.su.tutto}) the sequence $\{X_T^{n_j} \}_{j \in \mathbb{N} }$ is bounded in $\mathbb{L}^1(\mathbb{P})$, by Lemma 3.5 in \cite{K} there exist a nondecreasing, nonnegative, $\mathbb{R}^m$-valued c\`adl\`ag process $B$ defined on $[0,T]$ and a subsequence (not relabeled) of $\{ X^{n_j} \}_{j \in \mathbb{N} }$  such that, $\mathbb{P}$-a.s., 
\begin{equation}\label{kabanov}
\lim_m \int_{[0,T]} \varphi_t dB_t^m=\int_{[0,T]} \varphi_t dB_t  \quad \forall \, \varphi \in \mathcal{C}_b({[0,T]};\mathbb{R}^d) \quad \text{ and} \quad \lim_m B_T^m=B_T,
\end{equation}
where we have set, $\mathbb{P}$-a.s.
\begin{equation}\label{def_of_B^m}
B_t^m := \frac{1}{m}	\sum_{j=1}^m {X}_t^{n_j}, \quad\quad \forall	t \in {[0,T]}.
\end{equation}
Moreover, for $\varphi \in \mathcal{C}_c^\infty([0,T);\mathbb{R}^d)$, the limit in (\ref{kabanov}) and an integration by parts, together with the limit in (\ref{charatconv}) (observing that the sequence $\{ |X_T^n|\}_{n\in\mathbb{N}} $ is $\mathbb{P}$-a.s.\ bounded), imply that, $\mathbb{P}$-a.s., 
\begin{align*}
\int_{[0,T]} \varphi_t dB_t  = \lim_m \frac{1}{m} \sum_{j=1}^m \int_{[0,T]} \varphi_t dX_t^{n_j} = -\lim_m \frac{1}{m} \sum_{j=1}^m \int_0^T X_t^{n_j} \varphi_t'   dt = \int_{[0,T]} \varphi_t dX_t.
\end{align*}
Therefore, by the fundamental lemma of the Calculus of Variation (see Theorem 1.24 at p. 26 in \cite{Dacorogna}), the right-continuity of $X$ and $B$, and the convergence of $X_T^{n_j}$ to $X_T$, we have  $B_t=X_t$ for all $ t \in [0,T]$, $\mathbb{P}$-a.s. 
This identification allows to conclude, using (\ref{limit_Z}) and uniform integrability estimates as in (\ref{uniformintegraaplpha}), that
\begin{equation*}
\mathbb{E}[Z]= \lim_m \frac{1}{m} \sum_{j=1}^m \mathbb{E} \bigg[ \int_{[0,T]} F_t \, d X_t^{n_j} \bigg] = \lim_m \mathbb{E} \bigg[ \int_{[0,T]} F_t \, d {B}_t^{m} \bigg] = \mathbb{E} \bigg[ \int_{[0,T]} F_t \, d X_t \bigg]. 
\end{equation*}
The latter, combined with (\ref{limit_Z}), completes the proof of the lemma. 
\end{proof}

\smallskip
\textbf{Acknowledgements.} 
Financial support by the German Research Foundation (DFG) through the Collaborative Research Centre 1283 ``Taming uncertainty and profiting from randomness and low regularity in analysis, stochastics and their application'' is gratefully acknowledged.
We are grateful to Peter Bank, Markus Fischer, Ulrich Horst, Max Nendel, Frank Riedel, and Jan-Henrik Steg for fruitful conversations.

%%%%%%%%%%%%%%%%%%%%%%%%%%%%%%%%%%%%%%%%%%%%%%%%%%%%%%%%%%%%%%%%%%

%%%%%%%%%%%%%%%%%%%%%%%%%%%%%%%%%%%%%%%%%%%%%%%%%%%%

\end{document}